\documentclass[10pt, twoside]{article}
\usepackage{amsmath,amsthm,amssymb}
\usepackage{times}
\usepackage{enumerate}
\usepackage{graphicx}
\usepackage[noadjust]{cite}

\usepackage{color}

\def\titlerunning#1{\gdef\titrun{#1}}
\makeatletter
\def\author#1{\gdef\autrun{\def\and{\unskip, }#1}\gdef\@author{#1}}

\makeatother

\def\keywords#1{\par\medskip
\noindent\textbf{Keywords.} #1}
\def\subjclass#1{\par\smallskip
\noindent\textbf{MSC (2010):} #1}

\newtheorem{thm}{Theorem}[section]
\newtheorem{cor}[thm]{Corollary}
\newtheorem{lem}[thm]{Lemma}

\newtheorem{prop}[thm]{Proposition}

\theoremstyle{definition}
\newtheorem{defin}[thm]{Definition}
\newtheorem{rem}[thm]{Remark}
\newtheorem{exa}[thm]{Example}

\numberwithin{equation}{section}

\newtheorem*{notations}{Notations}

\DeclareMathOperator*{\esssup}{ess\,sup}

\makeatletter
\let\@fnsymbol\@alph
\makeatother

\begin{document}

\baselineskip=17pt

\titlerunning{Radial quasilinear problems}

\title{Existence results for a class of quasilinear Schr\"{o}dinger equations with singular or vanishing potentials
}

\author{Marino Badiale\thanks{Dipartimento di Matematica ``Giuseppe Peano'', Universit\`{a} degli Studi di
Torino, Via Carlo Alberto 10, 10123 Torino, Italy. 
e-mails: \texttt{marino.badiale@unito.it}, \texttt{michela.guida@unito.it}}
\textsuperscript{,}\thanks{Partially supported by the PRIN2012 grant ``Aspetti variazionali e
perturbativi nei problemi differenziali nonlineari''.}
\ -\ Michela Guida\textsuperscript{a}
\ -\ Sergio Rolando\thanks{Dipartimento di Matematica e Applicazioni, Universit\`{a} di Milano-Bicocca,
Via Roberto Cozzi 53, 20125 Milano, Italy. e-mail: \texttt{sergio.rolando@unito.it}}
}

\date{
}
\maketitle

\begin{abstract}
Given two continuous functions $V\left(r \right)\geq 0$ and $K\left(r\right)> 0$ ($r>0$), which may be singular or vanishing at zero as well as at infinity, we study the quasilinear elliptic equation
\begin{equation*}
-\Delta w+ V\left( \left| x\right| \right) w - w \left( \Delta w^2 \right)= K(|x|) g(w) \quad \text{in }\mathbb{R}^{N},  
\end{equation*}
\noindent where $N\geq3$. To study this problem we apply a change of variables $w=f(u)$, already used by several authors, and find existence results for nonnegative solutions by the application of variational methods. 
The main features of our results are that they do not require any compatibility between how the potentials $V$ and $K$ behave at the origin and at infinity, and that they essentially rely on power type estimates of the relative growth of $V$ and $K$, not of the potentials separately.
Our solutions satisfy a weak formulations of the above equation, but we are able to prove that they are in fact classical solutions in $\mathbb{R}^{N} \backslash \{  0\}$. 
To apply variational methods, we have to study the compactness of the embedding of a suitable
function space into the sum of Lebesgue spaces $L_{K}^{q_{1}}+L_{K}^{q_{2}}$, and thus into $L_{K}^{q}$ ($=L_{K}^{q}+L_{K}^{q}$) as a particular case. 
 The nonlinearity $g$ has a double-power behavior, whose standard example is $g(t) = \min \{ t^{q_1 -1}, t^{q_2 -1}  \}$, recovering the usual case of a single-power behavior when $q_1 = q_2$.



\end{abstract}

\section{Introduction}

In the present paper, we study the following quasilinear elliptic equation
\begin{equation}\label{EQ}
-\Delta w+ V\left( \left| x\right| \right) w - w \left( \Delta w^2 \right)= K(|x|) g(w)  \quad \text{in }\mathbb{R}^{N}
\end{equation}
where $N\geq3$, $V\geq 0$ and $K> 0$ are given potentials, and $g:\mathbb{R}\rightarrow\mathbb{R}$ is a continuous nonlinearity such that $g(0)=0$. Searching for standing waves solutions, this equation derives from an evolution Schr\"odinger equation which has been used to study several physical phenomena (see \cite{LiuWang,PoppenbergSchmittWang, Kwon} and the references therein), such as laser beams in matter \cite{Brandi-et} and quasi-solitons in superfluids films \cite{Kuri}.

It is not easy to apply variational methods to study (\ref{EQ}), because the (formally) associated functional  presents unusual integral terms, like $\int_\mathbb{R^N} w^2 |\nabla w |^2 dx$. In recent times, a great amount of work has been made on equation (\ref{EQ}) and several techniques have been introduced to overcome these difficulties (see  \cite{AiresSouto,ColinJeanjean,OMiyagakiSoares,Uberlandio1,Uberlandio3,Uberlandio2,
Uberlandio4,FangSzulkin,FurtadoSilvaSilva,Gloss,LiuLiuWang1,LiuLiuWang2,LiuWangWang1,LiuWangWang2, RuizSiciliano, SilvaVieira,YangDing,YangWangZhao,Li-Huang,Kwon} and the references therein). In this paper, following an idea introduced in \cite{LiuWang}, we exploit a change of variable $w=f(u)$ where $f$ satisfies a suitable ordinary differential equation (see Section 2). The problem in the new unknown $u$ can be faced with usual variational methods, working in an Orlicz-Sobolev space. This idea has been used in \cite{Uberlandio1, Uberlandio2, Uberlandio4, Kwon}, among others. 

In almost all the papers dealing with (\ref{EQ}), the potential $V$ (be it radial or nonradial) is supposed to be positive and bounded away from zero at infinity. At the best of our knowledge, the only papers dealing with a potential $V$ allowed to vanish at infinity are 
\cite{AiresSouto,Uberlandio2,Kwon,Li-Huang} (see also \cite{Uberlandio4} for equation \eqref{EQ} in presence of a parameter). 
In \cite{AiresSouto} and \cite{Kwon}, the authors respectively prove existence and nonexistence results assuming that $V$ is bounded.  
In \cite{Li-Huang}, existence of solutions is obtained for possibily singular $V$'s but bounded $K$'s.
In \cite{Uberlandio2}, which is the paper that inspired our work, both $V$ and $K$ can be singular or vanishing at zero or at infinity, and the authors prove existence of solution assuming that the potentials are radial and essentially behave as powers of $|x|$ as $|x|\to 0$ and $|x|\to \infty$ (see the paper introduction for more precise assumptions). 

Here we study equation (\ref{EQ}) via the change of variable $w=f(u)$ in the case in which both $V$ and $K$ are radial potentials that may be singular or vanishing at zero as well as at infinity.  
This implies that, even in the new variational setting brought in by the variable change, the usual embeddings theorems for Sobolev spaces are not available, and new embedding theorems need to be proved. We observe that, for semilinear and $p$-laplacian elliptic equations, this has been done in several papers: see e.g. the references in 
\cite{BGR_I,BGR_II,GR-nls} for a bibliography concerning the usual Laplace equation, \cite{Anoop,Su12,Cai-Su-Sun,SuTian12,Su-Wang-Will-p,Yang-Zhang,Zhang13,BPR,BGR_p}
for equations involving the $p$-laplacian, and \cite{BZ, Su-Wang} for problems with a potential $A$ on the derivatives (see also \cite{BGR_bilap} for biharmonic equations).

The main novelty in our approach (with respect to the previous literature, and especially to \cite{Uberlandio2}) is two-folded. 
First, we look for embeddings of a suitable function space not into a single (weighted) Lebesgue space $L_{K}^{q}$ but into a sum of Lebesgue spaces $L_{K}^{q_1}+L_{K}^{q_2}$. This allows to study separately the behaviour of the potentials $V$ and $K$ at $0$ and $\infty$, assuming independent sets of hypotheses about these behaviours. 
Second, we assume hypotheses not on $V$ and $K$ separately but on their ratio, so admitting asymptotic behaviors of general kind for the two potentials, not only power-like
(cf. Section 8).

As a conclusion, our approach shows that, in order to have solutions, the potentials $V$ and $K$ can have independent behaviours at zero and at infinity, no needing to satisfy compatibility conditions between such behaviours. Moreover, what does really count are not the growths of $V$ and $K$ separately, but only how they grow (or decay) relatively to one another.

The paper is organized as follows. In Section 2 we introduce our hypotheses on $V$ and $K$, the change of variables $w=f(u)$ and the main function spaces $X$ and $E$ we will work in. 
In Section \ref{COMP} we state a general result concerning the embedding properties of $E$  into $L_{K}^{q_{1}}+L_{K}^{q_{2}}$ (Theorem \ref{THM(cpt)})
and some explicit conditions ensuring that the embedding is compact
(Theorems \ref{THM0} and \ref{THM1}). The general
result is proved in Section \ref{SEC:1}, the explicit conditions in Section \ref{SEC:2}. 
In Section 6 we introduce our hypotheses on the nonlinearity $g$, we study the main properties of the functional $I$ associated to the dual problem and of its critical points, which give rise to solutions to (\ref{EQ}). In Section \ref{SEC: ex} we apply our embedding results to get existence of non negative solutions for equation (\ref{EQ}), stating and proving our main existence result, which is Theorem \ref{THM:ex}. Section \ref{SEC:EX} is devoted to concrete examples of potentials $V$ and $K$ satisfying our hypotheses, though escaping the previous literature.

\begin{notations}

We end this introductory section by collecting
some notations used in the paper.

\noindent  $\bullet $ $\mathbb{R}_{+} = ( 0, +\infty ) = \left\{ x\in \mathbb{R} : x>0 \right\}$.

\noindent $\bullet $ For every $R>0$, we set $B_{R} =\left\{ x\in \mathbb{R}
^{N}:\left| x\right| <R\right\} $.


\noindent $\bullet $ For any subset $A\subseteq \mathbb{R}^{N}$, we denote $
A^{c}:=\mathbb{R}^{N}\setminus A$. If $A$ is Lebesgue measurable, $\left|
A\right| $ stands for its measure.

\noindent $\bullet $ $\hookrightarrow $ denotes \emph{continuous} embeddings.

\noindent $\bullet $ If $Y$ is a Banach space, $Y'$ is its dual.

\noindent $\bullet $ $C_{\mathrm{c}}^{\infty }(\Omega )$ is the space of the
infinitely differentiable real functions with compact support in the open
set $\Omega \subseteq \mathbb{R}^{N}$. If $\Omega$ has radial symmetry, $C_{\mathrm{c}, r}^{\infty }( \Omega )$ is the subspace 
of $C_{\mathrm{c}}^{\infty }(\Omega )$ made of radial functions.

\noindent $\bullet $ For any measurable set $A\subseteq \mathbb{R}^{N}$, $
L^{q}(A)$ and $L_{\mathrm{loc}}^{q}(A)$ are the usual real Lebesgue spaces.
If $\rho :A\rightarrow \Bbb{R}_{+}$ is a measurable function, then $%
L^{p}(A,\rho \left( z\right) dz)$ is the real Lebesgue space with respect to
the measure $\rho \left( z\right) dz$ ($dz$ stands for the Lebesgue measure
on $\mathbb{R}^{N}$). In particular, if $K:\Bbb{R}_{+}\rightarrow \Bbb{R}
_{+} $ is measurable, we denote $L_{K}^{q}\left( A\right) :=L^{q}\left(
A,K\left( \left| x\right| \right) dx\right) $.

\noindent $\bullet $ For $N\geq 3$, $2^{* }= \frac{2N}{N-2}$ is the critical exponent of Sobolev embeddings.

\end{notations}

\section{Hypotheses and preliminary results} \label{SEC: hp}

\par \noindent Throughout this paper, we assume $N\geq 3$ and the following hypothesis $\left( \mathbf{H}\right) $ on $V,K$:

\begin{itemize}

\item[$\left( \mathbf{H}\right) $]  $V:\mathbb{R}_{+}\rightarrow
\left[ 0,+\infty \right) $ and $K:\mathbb{R}_{+} \rightarrow \mathbb{R}_{+} $ are continuous, and there is $C>0$ such that for all $r \in (0,1)$ one has
$$
V(r) \leq \frac{C}{r^2} .
$$
\end{itemize}

We begin by introducing the function $f$ we need to define the Orlicz-Sobolev space in which we will work. This function is defined as the solution of the following Cauchy problem:
\begin{equation}
\label{eq:change}
\left\{ 
\begin{array}{ll}
f'(t)= \frac{1}{\sqrt{1+2f(t)^2}} \quad \quad  \text{in }\mathbb{R} \\ 

f(0)=0    

\end{array}
\right. 
\end{equation}

\noindent The following lemma gives the main properties of the solution of (\ref{eq:change}). For the proofs see \cite{Uberlandio1, Uberlandio2}.

\begin{lem}
\label{change2}
There is a unique solution $f\in C^{\infty}(\mathbb{R}, \mathbb{R})$ of (\ref{eq:change}). Such a solution is odd, strictly increasing, and surjective (hence invertible). Moreover, it satisifes the following properties:
%
%
%
\begin{itemize}

\item[(1)] $|f'(t)|\leq 1 $ for all $t \in \mathbb{R}$;

\item[(2)] $|f(t)|\leq |t| $ for all $t \in \mathbb{R}$;

\item[(3)] $f(t)/t \rightarrow 1  $ as $t \rightarrow 0$;

\item[(4)] $f(t)/\sqrt{t} \rightarrow 2^{1/4}  $ as $t \rightarrow +\infty$;

\item[(5)] $f(t)/2\leq t f'(t) \leq f(t)  \,$ for all $t \geq 0$;

\item[(6)] $|f(t)|\leq  2^{1/4} \sqrt{|t|} $ for all $t \in \mathbb{R}$;

\item[(7)] There is a constant $C_1 >0$ such that 
$$|f(t)| \geq C_1 |t| \quad\text{if} \, \, |t|\leq 1; \quad \quad |f(t)| \geq C_1 \sqrt{|t|} \quad\text{if} \, \, |t|\geq 1 ;$$ 

\item[(8)] There are two positive constants $c_1 , c_2$ such that $|t| \leq c_1 |f(t)| + c_2 f(t)^2$ for all $t \in \mathbb{R}$;

\item[(9)] $|f(t) f'(t)|\leq  \frac{1}{\sqrt{2}} $ for all $t \in \mathbb{R}$;

\item[(10)] The function $f(t)^2$ is strictly convex;

\item[(11)] There is a constant $C>0$ such that $f(2t)^2 \leq C f(t)^2$ for all $t \in \mathbb{R}$.

\end{itemize}                                  

\end{lem}

We now use the function $f$ to define a change of unknown: we call $w$ the solution of (\ref{EQ}) that we are looking for and we set $w= f(u)$, where $u$ is the new unknown, living in a suitable space that we are going to define. In this way, to get solutions $w$ to (\ref{EQ}) we will look for solutions $u$ to the equation
\begin{equation}\label{EQdual}
-\Delta u+ V\left( \left| x\right| \right) f(u) f' (u) = K(|x|) g(f(u)) f' (u) \quad \text{in }\mathbb{R}^{N},  
\end{equation}
which will be obtained as critical points of the following functional:
 \begin{equation}\label{funct1}
I(u) = \frac{1}{2}\int_{\mathbb{R}^{N}}|\nabla u |^2 dx + \frac{1}{2}\int_{\mathbb{R}^{N}}V(|x|) f(u)^2  dx-\int_{\mathbb{R}^{N}}K(|x|) G(f (u)) \,dx 
\end{equation}
The critical points of $I$ and their relations with solutions of (\ref{EQ}) will be studied in Section \ref{SEC: dual}, 
my means of the following hypotheses on the nonlinearity: $g: {\mathbb{R}} \rightarrow {\mathbb{R}}$ is a continuous function satisfying

\begin{itemize}

\item[${\bf \left( g_{1}\right) } $]  $\exists \theta >2$ such that $0\leq 2\theta G\left( t\right) \leq g\left( t\right) t$ for all $t\in \mathbb{R}$;

\item[${\bf  \left( g_{2}\right)} $]  $\exists t_{0}>0$ such that $G\left(t_{0}\right) >0$, where $G(t)= \int_{0}^{t} g(s)ds $;

\item[$  \left( {\bf g}_{q_{1},q_{2}}\right) $]  there exits a constant $C>0$ such that
$\left|g\left( t\right) \right| \leq C \min \left\{ \left| t\right|^{q_{1}-1},\left| t\right| ^{q_{2}-1}\right\} $ for all $t\in \mathbb{R}$.

\end{itemize}

\noindent We notice that these hypotheses imply $q_1, q_2 \geq 2\theta $. We also observe that, if $q_{1}\neq q_{2}$, the double-power growth
condition $  \left( {\bf g}_{q_{1},q_{2}}\right) $ is more stringent than the more
usual single-power one, since it implies $|g(t)| \leq C |t|^{q-1}$
for $q=q_{1}$, $q=q_{2}$ and every $q$
in between. On the other hand, we will never require $q_{1}\neq q_{2}$ in $  \left( {\bf g}_{q_{1},q_{2}}\right) $, so that our results will also concern
single-power nonlinearities as long as we can take $q_{1}=q_{2}$.

In this section and in the following ones, we introduce the function space $E$ in which we will obtain critical points of $I$ and we study the relevant compactness results for $E$.

First, we introduce the space $D_r^{1,2} \left(\mathbb{R}^{N}  \right) $, which is the closure of $C_{\mathrm{c}, r}^{\infty }( \mathbb{R}^{N} ) $ with respect to the norm $||u||_{1,2} := \left( \int_{\mathbb{R}^{N}} |\nabla u|^2 
dx \right)^{1/2}$. It is well known that $D_r^{1,2} \left(\mathbb{R}^{N}  \right) $is a Hilbert space. Then we define a second Hilbert space
$$X := \left\{ u \in  D_r^{1,2} \left(\mathbb{R}^{N}\right) \, \Big| \, \int_{\mathbb{R}^{N}} V(|x|) |u|^2 dx <+\infty \right\}$$
endowed with the norm $||u|| := \left( ||u||_{1,2}^{2} + ||u||_{L^{2}(\mathbb{R}^{N}, V (|x|) dx )}^{2} \right)^{1/2}$. Finally we  introduce the main function space that we will use, which is
$$E := \left\{ u \in  D_r^{1,2} \left(\mathbb{R}^{N}\right) \, \Big| \, \int_{\mathbb{R}^{N}} V(|x|) f(u)^2 dx <+\infty \right\}.$$
In $E$, we first define the norm
$$
||u||_o := \inf_{k>0} \frac{1}{k} \left[ 1+  \int_{\mathbb{R}^{N}} V(|x|) f( k u)^2 dx    \right] ,
$$
which is an Orlicz norm. Then we introduce the norm
$$||u || := ||u||_{1,2} + ||u||_o .$$
The space $E$, endowed with the norm $|| . ||$, is an Orlicz-Sobolev space. In the results, we recall its main properties.

\begin{thm}
$\left( E, ||.||\right)$ is a Banach space and the following continuous embedding holds:
$$
E\hookrightarrow D_r^{1,2} \left(\mathbb{R}^{N}\right) .
$$
\end{thm}

\begin{proof}
The fact that $E$ is a Banach space derives from the general theory of Orlicz spaces, together with the properties of the function $f$ stated in the above lemma, in particular $(10)$ and $(11)$ (see \cite{Uberlandio2}). The embedding is obvious from the definitions of $E$ and its norm. 
\end{proof}

\begin{cor}
\label{COR:est}
There are constants $S_N , C_N >0$ (only depending on $N$) such that for all $u \in E$ it holds:
$$
\left( \int_{\mathbb{R}^{N}}  |u|^{2^*} dx \right)^{1/2^*}\leq S_N \, ||u||, \quad \quad |u(x) |\leq C_N \frac{||u||}{|x|^{\frac{N-2}{2}}} \quad a.e. \, x \in \mathbb{R}^{N}.
$$
\end{cor}
\begin{proof}
These are well known properties of any $u \in D_r^{1,2} \left(\mathbb{R}^{N}\right) $.
\end{proof}

\begin{lem}
\label{lem:properties}

\begin{itemize}

\item[(1)]

There exists $C>0 $ such that for all $u \in E$ one has
$$\frac{\int_{\mathbb{R}^{N}} V(|x|) f(u)^2  dx }{1 + \left( \int_{\mathbb{R}^{N}} V(|x|) f( u)^2  dx \right)^{1/2}}\leq C \, ||u||.$$

\item[(2)] If $u_n \rightarrow u$ in $E$, then
$$
\int_{\mathbb{R}^{N}} V(|x|) \left| f(u_n )^2  - f(u)^2  \right| \, dx \rightarrow 0
\quad
\text{and}
\quad
\int_{\mathbb{R}^{N}} V(|x|) \left| f (u_n ) - f (u) \right|^2 \, dx \rightarrow 0.
$$

\item[(3)] If $u_n (x) \rightarrow u(x)$ a.e. in $\mathbb{R}^{N}$ and 
$$
\int_{\mathbb{R}^{N}} V(|x|) f(u_n )^2  \, dx \rightarrow \int_{\mathbb{R}^{N}} V(|x|)  f(u )^2  \, dx 
$$
then $||u_n - u  ||_o  \rightarrow 0$.

\end{itemize}

\end{lem}
\begin{proof}
For the proof see \cite{LiuWangWang1}, \cite{Uberlandio1}, \cite{Uberlandio2}. We just point out that the proofs   also work in our hypotheses, which are a little different from theirs. 
\end{proof}


\begin{cor}
\label{COR:embed}
Assuming ${\bf (H})$, we have the continuous embedding $X \hookrightarrow E$.
\end{cor}

\proof
This is Corollary 2.1 of \cite{Uberlandio1}. Their proof can be repeated in our case.
\endproof

\begin{lem}
$C_{\mathrm{c}, r}^{\infty }( \mathbb{R}^{N} ) $ is dense in $E$.
\end{lem}
\begin{proof}
The proof is contained in the Master's degree thesis \cite{Ligani}, which is unpublished and in Italian, so we give it here. Let $u \in E$ and assume first that 
supp$\,u$ is bounded. By standard results, there is a sequence $ \{ \varphi_n \}_n \subseteq C_{\mathrm{c}, r}^{\infty }( \mathbb{R}^{N} ) $ such that $\varphi_n  \rightarrow u$ in $D_r^{1,2} \left(\mathbb{R}^{N}\right)$ and $\varphi_n  (x) \rightarrow u(x)$ for a.e. $x\in\mathbb{R}^{N} $. We have to prove that $||\varphi_n - u ||_o \rightarrow 0$. Thanks to Lemma \ref{lem:properties}, it is enough to prove the following claim:
$$
\int_{\mathbb{R}^{N}} V(|x|) f (\varphi_n )^2\, dx \rightarrow \int_{\mathbb{R}^{N}} V(|x|)  f(u )^2 \, dx .
$$
By hypothesis, there is $M>0$ such that supp$u \subseteq B_M$. As the sequence $ \{ \varphi_n \}_n $ is obtained by convolution of $u$ with a family of mollifiers, we can assume supp$\varphi_n \subseteq B_{M+1}$ for all $n$. Hence we have
$$
\int_{\mathbb{R}^{N}} V(|x|) f(\varphi_n)^2 \, dx = \int_{\{ |x| \leq 1 \} } V(|x|)  f(\varphi_n )^2 \, dx + \int_{\{ 1 \leq |x| \leq M+1 \} } V(|x|)  f (\varphi_n)^2 \, dx.
$$
We have $V(|x|) f(\varphi_n (x))^2 \rightarrow V(|x|) f(u (x))^2 $ a.e., and we will apply Dominated Convergence Theorem. If $|x|\leq 1$ and $x \not= 0$, using ${\bf (H)} $ and Lemma \ref{change2} we have
$$
V(|x|) f(\varphi_ n (x))^2 \leq C \frac{f(\varphi_ n (x))^2 }{|x|^2} \leq C \frac{\varphi_ n (x)^2 }{|x|^2}.
$$
As $\varphi_n  \rightarrow u$ in $D_r^{1,2} \left(\mathbb{R}^{N}\right)$, from Hardy's inequality we get $\frac{\varphi^2_ n  }{|x|^2} \rightarrow 
\frac{u^2 }{|x|^2}$ in $L^1 ( \mathbb{R}^{N} )$, whence, up to a subsequence, there exists a function $h \in L^1 ( \mathbb{R}^{N} )$ such that 
$\frac{\varphi^2_n  }{|x|^2} \leq h$. This implies 
$$
V(|x|) f(\varphi_ n (x))^2 \leq C h
$$
for a.e. $x \in B_1$. By Dominated Convergence Theorem we get
$$
\int_{\{ |x| \leq 1 \} } V(|x|)  f(\varphi_n )^2 \, dx \rightarrow \int_{\{ |x| \leq 1 \} } V(|x|)  f(u )^2 \, dx .
$$
If $1 \leq |x| \leq M+1$ we get $V(|x|) f(\varphi_n )^2 \leq C \varphi_n^2 $. As $\varphi_n  \rightarrow u$ in $L^{2^*} \left(\mathbb{R}^{N}\right)$, we 
have  $\varphi_n  \rightarrow u$ in $L^{2} \left( B_{M+1}\right)  $, whence there exists $h_1 \in L^1 \left( B_{M+1}\right) $ such that $\varphi_n (x)^2 \leq h_1 (x)$ for a.e. $x \in B_{M+1}$. Hence $V(|x|) f(\varphi_n)^2 \leq C h_1$ and by Dominated Convergence
$$
\int_{\{ 1 \leq |x| \leq M+1 \} } V(|x|)  f(\varphi_n )^2 \, dx \rightarrow \int_{\{ 1 \leq |x| \leq M+1 \} } V(|x|)  f(u )^2 \, dx.
$$
This concludes the proof if supp$\,u$ is bounded. In the general case, we choose a sequence of standard truncation functions $\{ \zeta_n \}_n $. It is easy to show that $\zeta_n u \rightarrow u$ in $E$ for every $u \in E$, and combining this result with the previous one we get the thesis.
\end{proof}

\begin{lem}
\label{A6}

For any $r,R $ such that $0<r<R$, the embedding
$$
E \hookrightarrow L^{2}(B_R \setminus {\overline B_r })
$$
is continuous and compact.
\end{lem}

\begin{proof}
The embedding result is easily proved for the space $D_r^{1,2} \left(\mathbb{R}^{N}\right)$, so the thesis derives from the continuous embedding 
$E\hookrightarrow D_r^{1,2} \left(\mathbb{R}^{N}\right).$ 
\end{proof}

\section{Compactness results for the space $E$}\label{COMP}

Let $N\geq 3$ and let $V$ and $K$ be as in $\left( \mathbf{H}\right) $.
In this section we state the main compactness results of this paper, concerning the space $E$, defined as above.  
The compactness results that we state here will be proved in Sections \ref{SEC:1} and \ref{SEC:2}. They concern the embedding properties of $E$ into the sum
space 
\[
L_{K}^{q_{1}}+L_{K}^{q_{2}}:=\left\{ u_{1}+u_{2}:u_{1}\in
L_{K}^{q_{1}}\left( \Bbb{R}^{N}\right) ,\,u_{2}\in L_{K}^{q_{2}}\left( \Bbb{R%
}^{N}\right) \right\} ,\quad 1<q_{i}<\infty . 
\]
We recall from \cite{BPR} that such a space can be characterized as the set
of measurable mappings $u:\Bbb{R}^{N}\rightarrow \Bbb{R}$ for which there
exists a measurable set $A\subseteq \Bbb{R}^{N}$ such that $u\in
L_{K}^{q_{1}}\left( A\right) \cap L_{K}^{q_{2}}\left( A^{c}\right) $. It is
a Banach space with respect to the norm 
\[
\left\| u\right\| _{L_{K}^{q_{1}}+L_{K}^{q_{2}}}:=\inf_{u_{1}+u_{2}=u}\max
\left\{ \left\| u_{1}\right\| _{L_{K}^{q_{1}}(\Bbb{R}^{N})},\left\|
u_{2}\right\| _{L_{K}^{q_{2}}(\Bbb{R}^{N})}\right\} 
\]
and the continuous embedding $L_{K}^{q}\hookrightarrow
L_{K}^{q_{1}}+L_{K}^{q_{2}}$ holds for all $q\in \left[ \min \left\{
q_{1},q_{2}\right\}, \max \left\{ q_{1},q_{2}\right\} \right] $. Our general embedding result is Theorem \ref{THM(cpt)} below. The assumptions of this result are quite general but not so easy to check, so more handy conditions ensuring these general assumptions will be provided by the next results. 

To state our results we introduce the following functions of $R>0$ and $q>1$:

\begin{eqnarray}
\mathcal{S}_{0}\left( q,R\right)&:=&
\sup_
{u\in E,\,
\left\| u\right\| =1  }
\int_{B_{R}}K\left( \left| x\right| \right)
\left| u\right| ^{q}dx,  \label{S_o :=}
\\
\mathcal{S}_{\infty }\left( q,R\right)&:=&
\sup_
{u\in E,\,
\left\| u\right\| =1  } 
\int_{\mathbb{R}%
^{N}\setminus B_{R}}K\left( \left| x\right| \right) \left| u\right| ^{q}dx.
\label{S_i :=}
\end{eqnarray}

\noindent  Clearly $\mathcal{S}_{0}\left( q,\cdot \right) $ is nondecreasing, $\mathcal{
S}_{\infty }\left( q,\cdot \right) $ is nonincreasing and both of them can
be infinite at some $R$.

\begin{thm}
\label{THM(cpt)} Let $q_{1},q_{2}>1$.

\begin{itemize}
\item[(i)]  If 
\begin{equation}
\mathcal{S}_{0}\left( q_{1},R_{1}\right) <\infty \quad \text{and}\quad 
\mathcal{S}_{\infty }\left( q_{2},R_{2}\right) <\infty \quad \text{for some }%
R_{1},R_{2}>0,  
\tag*{$\left( {\cal S}_{q_{1},q_{2}}^{\prime }\right) $}
\end{equation}
then $E$ is continuously embedded into $L_{K}^{q_{1}}(\mathbb{R}^{N})+L_{K}^{q_{2}}(\mathbb{R}^{N})$.

\item[(ii)]  If 
\begin{equation}
\lim_{R\rightarrow 0^{+}}\mathcal{S}_{0}\left( q_{1},R\right)
=\lim_{R\rightarrow +\infty }\mathcal{S}_{\infty }\left( q_{2},R\right) =0, 
\tag*{$\left({\cal S}_{q_{1},q_{2}}^{\prime \prime }\right) $}
\end{equation}
then $E$ is compactly embedded into $L_{K}^{q_{1}}(\mathbb{R}^{N})+L_{K}^{q_{2}}(\mathbb{R}^{N})$.

\end{itemize}
\end{thm}

\noindent It is obvious that $(\mathcal{S}_{q_{1},q_{2}}^{\prime \prime })$
implies $(\mathcal{S}_{q_{1},q_{2}}^{\prime })$. Moreover, these assumptions
can hold with $q_{1}=q_{2}=q$ and therefore Theorem \ref{THM(cpt)} also
concerns the embedding properties of $X$ into $L_{K}^{q}$, $1<q<\infty $.

We now look for explicit conditions on $V$ and $K$ implying $(\mathcal{S}
_{q_{1},q_{2}}^{\prime \prime })$ for some $q_{1}$ and $q_{2}$. More
precisely, in Theorem \ref{THM0} we will find a range of exponents $
q_{1} $ such that $\lim_{R\rightarrow 0^{+}}\mathcal{S}_{0}\left(q_{1},R\right)$ $=0$, while 
in Theorem \ref{THM1} we will do the same for exponents $q_{2}$ such that
$\lim_{R\rightarrow +\infty}\mathcal{S}_{\infty }\left( q_{2},R\right) =0$.


For $\alpha \in \mathbb{R}$, $\beta \in \left[ 0,1\right] $, we define three
functions $\alpha ^{*}\left( \beta \right) $, $q_0^{*} \left( \alpha ,\beta
\right) $, $q_{\infty}^{*}\left( \alpha ,\beta
\right)$ by setting 
$$
\alpha ^{*}\left( \beta \right) := \max \left\{  \beta \frac{N+2}{2} -1-\frac{N}{2} ,  \frac{\beta}{2} \left(3N -2 \right) -N  \right\},$$
$$q_0^{*}\left( \alpha ,\beta \right) :=\frac{2 \alpha  +2 N - \beta (N+2)}{N-2}, 
\quad\quad  q_{\infty}^{*}\left( \alpha ,\beta \right) :=2 \, \frac{ \alpha  + N - 2\beta }{N-2}.
$$
Notice that $\alpha ^{*}\left( \beta \right) =  \beta \frac{N+2}{2} -1-\frac{N}{2} = - \frac{N+2}{2} (1 - \beta )$ when $0\leq \beta \leq \frac{1}{2}$, 
and $\alpha ^{*}\left( \beta \right) =  \frac{\beta}{2} \left(3N -2 \right) -N $ when $\frac{1}{2}\leq \beta \leq 1 $.

\begin{thm}
\label{THM0}
Assume that there exists $R_{1}>0$ such that
\begin{equation}
\sup_{r\in \left( 0,R_{1}\right) }\frac{K\left( r\right) }{%
r^{\alpha _{0}}V\left( r\right) ^{\beta _{0}}}<+\infty \quad \text{for some }%
0\leq \beta _{0}\leq 1\text{~and }\alpha _{0}>\alpha ^{*}\left( \beta_{0}\right) .  \label{esssup in 0}
\end{equation}
Then $\displaystyle \lim_{R\rightarrow 0^{+}}\mathcal{S}_{0}\left(
q_{1},R\right) =0$ for every $q_{1}\in \mathbb{R}$ such that 
\begin{equation}
\max \left\{ 1,2\beta _{0}\right\} <q_{1}<q_0^{*}  \left( \alpha _{0},\beta
_{0}\right) .  \label{th1}
\end{equation}
\end{thm}

\bigskip

 \noindent Notice that, as $\beta \leq 1$, it holds $\alpha ^{*}\left( \beta \right) \geq -N (1- \beta )$. 
Also notice that the inequality $\max \left\{ 1,2\beta
_{0}\right\} <q_0^{*}\left( \alpha _{0},\beta _{0}\right) $ is equivalent to $
\alpha _{0}>\alpha ^{*}\left( \beta _{0}\right) $, so that 
such inequality is automatically true in (\ref{th1}) and does not ask for further
conditions on $\alpha _{0}$ and $\beta _{0}$.

\begin{thm}
\label{THM1}
Assume that there exists $R_{2}>0$ such that
\begin{equation}
\sup_{r>R_{2}}\frac{K\left( r\right) }{r^{\alpha _{\infty
}}V\left( r\right) ^{\beta _{\infty }}}<+\infty \quad \text{for some }0\leq
\beta _{\infty }\leq 1\text{~and }\alpha _{\infty }\in \mathbb{R}.
\label{esssup all'inf}
\end{equation}
Then $\displaystyle \lim_{R\rightarrow +\infty }\mathcal{S}_{\infty }\left(
q_{2},R\right) =0$ for every $q_{2}\in \mathbb{R}$ such that 
\begin{equation}
q_{2}>\max \left\{ 1,2\beta _{\infty },q_{\infty}^{*}\left(\alpha _{\infty },\beta
_{\infty }\right) \right\} .  \label{th2}
\end{equation}
\end{thm}

\begin{rem}   \label{RMK: suff12}
\begin{enumerate}
\item  \label{RMK: suff12-V^0}We mean $V\left( r\right) ^{0}=1$ for every $r$
(even if $V\left( r\right) =0$). In particular, if $V\left( r\right) =0$ for $r>R_{2}$, then Theorem \ref{THM1} can be applied with $\beta
_{\infty }=0$ and assumption (\ref{esssup all'inf}) means 
\[
\esssup_{r>R_{2}}\frac{K\left( r\right) }{r^{\alpha _{\infty }}}%
<+\infty \quad \text{for some }\alpha _{\infty }\in \mathbb{R}.
\]
Similarly for Theorem \ref{THM0} and assumption (\ref{esssup in 0}), if $%
V\left( r\right) =0$ for $r\in \left( 0,R_{1}\right) $.

\item  \label{RMK: suff12-Vbdd}The assumptions of Theorems \ref{THM0} and 
\ref{THM1} may hold for different pairs $\left( \alpha _{0},\beta_{0}\right)$,
$\left( \alpha _{\infty },\beta _{\infty }\right) $. In this
case, of course, one chooses them in order to get the ranges for $q_{1},q_{2}
$ as large as possible. For example, assume that $V$ is bounded in a neighbourhood of 0. 
If condition (\ref{esssup in 0}) holds true for a pair $\left( \alpha _{0},\beta _{0}\right) $, then 
(\ref{esssup in 0}) also holds for all pairs $\left( \alpha
_{0}^{\prime },\beta _{0}^{\prime }\right) $ such that $\alpha _{0}^{\prime
}<\alpha _{0}$ and $\beta _{0}^{\prime }<\beta _{0}$. Therefore, since $\max
\left\{ 1,2\beta \right\} $ is nondecreasing in $\beta $ and $q_0^{*}\left( \alpha ,\beta \right) $ is increasing in 
$\alpha $ and decreasing in $\beta $, it is convenient to choose $\beta _{0}=0$ and the best interval where one
can take $q_{1}$ is $1<q_{1}<q_0^{*}\left(\overline{\alpha },0\right) $ with $
\overline{\alpha }:=\sup \left\{ \alpha _{0}:\esssup_{r\in \left(
0,R_{1}\right) }K\left( r\right) /r^{\alpha _{0}}<+\infty \right\} $
(here we mean $q_0^{*}\left( +\infty ,0\right) =+\infty $).
\end{enumerate}
\end{rem}


\section{Proof of Theorem \ref{THM(cpt)} \label{SEC:1}}

In this section we assume, as usual, $N\geq 3$ and hypothesis $\left( \mathbf{H}\right) $.

\begin{lem}
\label{Lem(corone)}Let $R>r>0$ and $1<q<\infty $. 
Then there exist $\tilde{C}=\tilde{C}\left(N,r,R,q \right) >0$ and $l=l\left(q \right) >0$
such that $q-2l>0$ and $\forall u\in E$ one has 
\begin{equation}
\int_{B_{R}\setminus B_{r}}K\left( |x| \right) \left|u\right| ^q dx
\leq
\tilde{C}\left\| K \right\|_{L^{\infty}(B_{R}\setminus B_{r})}
\left\|u\right\|^{q-2l} \left(\int_{B_R\setminus B_r}\left|u\right|^{2}dx\right)^l.
\label{LEM:corone}
\end{equation}
\end{lem}

\proof
Let $u\in E$ and fix $t>1$ such that $t'q>2$ (where $t'=t/(t-1)$). 
Then, by H\"{o}lder inequality and the pointwise estimates of Corollary \ref{COR:est}, we have 
\begin{eqnarray*}
\int_{B_{R}\setminus B_{r}}K\left( \left| x\right| \right) \left| u\right|^{q} dx 
&\leq & \left( \int_{B_{R}\setminus B_{r}}K\left(\left| x\right| \right) ^{t}dx\right) ^{\frac{1}{t}}
\left( \int_{B_{R}\setminus B_{r}}\left|u\right| ^{t^{\prime }q}dx\right) ^{\frac{1}{t^{\prime }}} \\
&\leq & 
\left| B_{R}\setminus B_{r}\right| ^{\frac{1}{t}}
\left\| K  \right\|_{L^{\infty}(B_{R}\setminus B_{r})}
\left( \int_{B_{R}\setminus B_{r}}\left|u\right| ^{t^{\prime }q-2}\left|u\right|^{2} dx\right) ^{\frac{1}{t^{\prime }}}
\\
&\leq &
\left| B_{R}\setminus B_{r}\right| ^{\frac{1}{t}}
\left\| K  \right\|_{L^{\infty}(B_{R}\setminus B_{r})}
\left( \frac{C_N \left\|u\right\| }{r^{\frac{N-2}{2}}}\right) ^{q-2/t^{\prime}}
\left( \int_{B_{R}\setminus B_{r}}\left|u\right|^{2} dx\right) ^{\frac{1}{t^{\prime }}}.
\end{eqnarray*}
This proves (\ref{LEM:corone}), setting $l= 1/t'$ and $\tilde{C}= \left| B_{R}\setminus B_{r}\right| ^{\frac{1}{t}}
\left( C_N r^{-(N-2)/2}\right) ^{q-2/t^{\prime}}$.
\endproof

We now prove Theorem \ref{THM(cpt)}. Recall the definitions (\ref{S_o :=})-(\ref{S_i :=}) of the functions $\mathcal{S}_{0}$ and 
$\mathcal{S}_{\infty }$, and the following result from \cite{BPR} concerning convergence in the sum of Lebesgue spaces.

\begin{prop}[{\cite[Proposition 2.7]{BPR}}] 
\label{Prop(->0)}
Let $\left\{ u_{n}\right\}
\subseteq L_{K}^{p_{1}}+L_{K}^{p_{2}}$ be a sequence such that $\forall
\varepsilon >0$ there exist $n_{\varepsilon }>0$ and a sequence of
measurable sets $E_{\varepsilon ,n}\subseteq \mathbb{R}^{N}$ satisfying 
\begin{equation}
\forall n>n_{\varepsilon },\quad \int_{E_{\varepsilon ,n}}K\left( \left|
x\right| \right) \left| u_{n}\right| ^{p_{1}}dx+\int_{E_{\varepsilon
,n}^{c}}K\left( \left| x\right| \right) \left| u_{n}\right|
^{p_{2}}dx<\varepsilon .  \label{Prop(->0): cond}
\end{equation}
Then $u_{n}\rightarrow 0$ in $L_{K}^{p_{1}}+L_{K}^{p_{2}}$.
\end{prop}

\proof[Proof of Theorem \ref{THM(cpt)}]
We prove each part of the theorem separately.\smallskip

\noindent (i) By the monotonicity of $\mathcal{S}_{0}$ and $\mathcal{S}%
_{\infty }$, it is not restrictive to assume $R_{1}<R_{2}$ in hypothesis $%
\left( \mathcal{S}_{q_{1},q_{2}}^{\prime }\right) $. In order to prove the
continuous embedding, let $u\in E$, $u\neq 0$. Then we
have 
\begin{equation}
\int_{B_{R_{1}}}K\left( \left| x\right| \right) \left| u\right|
^{q_{1}}dx=\left\| u\right\| ^{q_{1}}\int_{B_{R_{1}}}K\left( \left| x\right|
\right) \frac{\left| u\right| ^{q_{1}}}{\left\| u\right\| ^{q_{1}}}dx\leq
\left\| u\right\| ^{q_{1}}\mathcal{S}_{0}\left( q_{1},R_{1}\right) 
\label{pf1}
\end{equation}
and, similarly, 
\begin{equation}
\int_{B_{R_{2}}^{c}}K\left( \left| x\right| \right) \left| u\right|
^{q_{2}}dx\leq \left\| u\right\| ^{q_{2}}\mathcal{S}_{\infty }\left(
q_{2},R_{2}\right) .  \label{pf2}
\end{equation}
We now use (\ref{LEM:corone}) of Lemma \ref{Lem(corone)} and Lemma \ref{A6} to deduce that there exists a constant $\tilde{C}_{1}>0$, 
independent from $u$, such that 
\begin{equation}
\int_{B_{R_{2}}\setminus B_{R_{1}}}K\left( \left| x\right| \right) \left|
u\right| ^{q_{1}}dx\leq \tilde{C}_{1}\left\| u\right\| ^{q_{1}}.  \label{pf3}
\end{equation}
Hence $u\in L_{K}^{q_{1}}(B_{R_{2}})\cap L_{K}^{q_{2}}(B_{R_{2}}^{c})$ and
thus $u\in L_{K}^{q_{1}}+L_{K}^{q_{2}}$. Moreover, if $u_{n}\rightarrow 0$
in $E$, then, using (\ref{pf1}), (\ref{pf2}) and (\ref
{pf3}), we get 
\[
\int_{B_{R_{2}}}K\left( \left| x\right| \right) \left| u_{n}\right|
^{q_{1}}dx+\int_{B_{R_{2}}^{c}}K\left( \left| x\right| \right) \left|
u_{n}\right| ^{q_{2}}dx=o\left( 1\right) _{n\rightarrow \infty },
\]
which means $u_{n}\rightarrow 0$ in $L_{K}^{q_{1}}+L_{K}^{q_{2}}$ by
Proposition \ref{Prop(->0)}. \emph{\smallskip }

\noindent (ii) Assume hypothesis $\left( \mathcal{S}_{q_{1},q_{2}}^{\prime
\prime }\right) $. Let $\varepsilon >0$ and let $u_{n}\rightharpoonup 0$ in $
E$. Then $\left\{ \left\| u_{n}\right\| \right\}_n $ is
bounded and, arguing as for (\ref{pf1}) and (\ref{pf2}), we can take $
r_{\varepsilon }>0$ and $R_{\varepsilon }>r_{\varepsilon }$ such that for
all $n$ one has 
\[
\int_{B_{r_{\varepsilon }}}K\left( \left| x\right| \right) \left|
u_{n}\right| ^{q_{1}}dx\leq  
\left\| u_{n}\right\| ^{q_{1}} \, \mathcal{S}_{0}\left( q_{1},r_{\varepsilon }\right) \leq \left( \sup_{n}\left\| u_{n}\right\|
^{q_{1}}\right) \mathcal{S}_{0}\left( q_{1},r_{\varepsilon }\right) <\frac{
\varepsilon }{3}
\]
and 
\[
\int_{B_{R_{\varepsilon }}^{c}}K\left( \left| x\right| \right) \left|
u_{n}\right| ^{q_{2}}dx\leq \left( \sup_{n}\left\| u_{n}\right\| ^{q_{2}} \right) \mathcal{S}
_{\infty }\left( q_{2},R_{\varepsilon }\right) <\frac{\varepsilon }{3}.
\]
Using (\ref{LEM:corone}) of Lemma \ref{Lem(corone)} and the boundedness of $\left\{ \left\|
u_{n}\right\| \right\} $ again, we infer that there exist two constants $
\tilde{C}_{2},l>0$, independent from $n$, such that 
\[
\int_{B_{R_{\varepsilon }}\setminus B_{r_{\varepsilon }}}K\left( \left|
x\right| \right) \left| u_{n}\right| ^{q_{1}}dx\leq \tilde{C}_{2}\left(
\int_{B_{R_{\varepsilon }}\setminus B_{r_{\varepsilon }}}\left| u_{n}\right|
^{2}dx\right) ^{l},
\]
where 
\[
\int_{B_{R_{\varepsilon }}\setminus B_{r_{\varepsilon }}}\left| u_{n}\right|
^{2}dx\rightarrow 0\quad \text{as }n\rightarrow \infty \quad \text{(}
\varepsilon ~\text{fixed)}
\]
thanks to Lemma \ref{A6}. Therefore we obtain 
\[
\int_{B_{R_{\varepsilon }}}K\left( \left| x\right| \right) \left|
u_{n}\right| ^{q_{1}}dx+\int_{B_{R_{\varepsilon }}^{c}}K\left( \left|
x\right| \right) \left| u_{n}\right| ^{q_{2}}dx<\varepsilon 
\]
for all $n$ sufficiently large, which means $u_{n}\rightarrow 0$ in $%
L_{K}^{q_{1}}+L_{K}^{q_{2}}$ (Proposition \ref{Prop(->0)}). This concludes
the proof of part (ii).
\endproof


\section{Proof of Theorems \ref{THM0} and \ref{THM1} \label{SEC:2}}

Assume as usual $N \geq 3$ and hypothesis $\left( \mathbf{H}\right)$. 
\begin{lem}
\label{Lem(Omega)}Let $R_0>0$ and assume
\[
\Lambda :=\sup_{x\in B_{R_0} }\frac{K\left( \left| x\right|
\right) }{\left| x\right| ^{\alpha }V\left( \left| x\right| \right) ^{\beta }%
}<+\infty \quad \text{for some }0\leq \beta \leq 1\text{~and }\alpha \in 
\mathbb{R}.
\]
Let $u\in E$ and assume that there exist $\nu \in \mathbb{R}$ and $m>0$
such that 
\[
\left| u\left( x\right) \right| \leq \frac{m}{\left| x\right| ^{\nu }}\quad 
\text{almost everywhere in } B_{R_0} .
\]
Then there exists a constant $C=C(N, R_0, \alpha, \beta )>0$ such that $\forall R \in (0,R_{0})$ and $\forall q>\max \left\{ 1,2\beta \right\} $, one has 
\medskip

$\displaystyle\int_{B_R }K\left( \left| x\right| \right) \left| u\right| ^{q} dx$
\[
\leq \left\{ 
\begin{array}{ll}
\Lambda  C m^{q-1}\left( \int_{B_R}|x|^{\frac{\alpha - \nu (q-1)}{N+2}2N } dx \right)^{\frac{N+2}{2N}} \, ||u|| & \text{if }\beta = 0, \\ 
\Lambda  C \left[ m^{q-1} \left( \int_{B_R}\left| x\right| ^{\frac{
\alpha -\nu \left( q-1\right) }{N+2 - \beta \left( N+2 \right) }2N}dx\right)^{
\frac{N+2 - \beta \left( N+2 \right) }{2N}} || u||^{1-\beta} + R^{N (1-\beta)+ \alpha }\right] \left( \int_{B_R} V(|x|) f(u)^2dx \right)^{\beta} \quad \medskip  & 
\text{if }0 < \beta < \frac{1}{2}, \\ 
 \Lambda  C \left[ m^{q-\beta }\left( \int_{B_R }\left| x\right|^{\frac{\alpha
-\nu \left( q-\beta \right) }{1-\beta }}dx\right) ^{1-\beta }  + R^{N (1-\beta )+ \alpha} \right] \left( \int_{B_R} V(|x|) f(u)^2dx \right)^{\beta},   \\
\medskip  & \text{if }\frac{1}{2}
\leq \beta <1, \\ 
\Lambda C \left[  m^{q-1} \left( \int_{B_R }\left| x\right| ^{2\alpha -2 \nu \left(
q-1\right) }V\left( | x| \right) f(u)^2  dx\right) ^{
\frac{1}{2}} \left( \int_{B_R} V(|x|) f(u)^2dx \right)^{\frac{1}{2}} + R^{\alpha}\int_{B_R} V(|x|) f(u)^2dx  \right] & \text{if }\beta =1.
\end{array}
\right. 
\]
\end{lem}

\begin{proof}

Let us take $ R \in (0,R_{0})$ and define
$$B_{R}^1 = B_R \cap \{ x \in \mathbb{R}^{N} \, | \, |u(x)| \geq 1 \}, \quad B_{R}^2 = B_R \cap \{ x \in \mathbb{R}^{N} \, | \, |u(x) | < 1 \}.$$
Recall that, by Lemma \ref{change2}, there is $C_1 >0$ such that $|f(t)| \geq C_1 |t|$ when $|t|\leq 1$ and $|f(t)| \geq C_1 |t|^{1/2}$ when $|t| \geq 1$ . This implies 
$|f(u(x))| \geq C_1 |u(x)|^{1/2}$ when $x\in B_{R}^1$ and $|f(u(x))| \geq C_1 |u(x)|$ when $x\in B_{R}^2$, whence
\begin{equation}\label{disf}
\int_{B_{R}^1}V(|x|) f(u)^2  dx \geq C_{1}^2 \int_{B_{R}^1}V(|x|) |u|dx, \quad \int_{B_{R}^2}V(|x|) f(u)^2  dx \geq C_{1}^2 \int_{B_{R}^2}V(|x|) |u|^2 dx
\end{equation}
\noindent We distinguish several cases, where we will use H\"{o}lder inequality many
times. \smallskip

\noindent \emph{Case }$\beta =0$.\ \ 
\noindent We apply H\"older inequality with exponents $2^* = \frac{2N}{N-2}$ and $\frac{2N}{N+2}$, and standard Sobolev inequality (Corollary \ref{COR:est}), in order to get
{\allowdisplaybreaks
\begin{eqnarray*}
\frac{1}{\Lambda }\int_{B_R }K\left( \left| x\right| \right) \left|
u\right| ^{q} dx 
&\leq & \int_{B_R }\left| x\right|^{\alpha }\left| u\right| ^{q-1}\left| u\right| dx \\
&\leq & \left( \int_{B_R
}\left( \left| x\right| ^{\alpha }\left| u\right| ^{q-1}\right) ^{\frac{2N}{
N+2}}dx\right) ^{\frac{N+2}{2N}}\left( \int_{B_R }\left| u\right|
^{2^*}dx\right) ^{\frac{1}{2^*}} \\
&\leq & m^{q-1} S_N \left( \int_{B_R }\left| x\right| ^{\frac{\alpha -\nu
\left( q-1\right) }{N+2}2N}dx\right) ^{\frac{N+2}{2N}}\left\| u\right\| .
\end{eqnarray*}
}

\noindent \emph{Case }$0<\beta <1/2$.\ \ 
\noindent We write
$$\frac{1}{\Lambda }\int_{B_R }K\left( |x| \right) \left|
u\right|^{q} dx =\frac{1}{\Lambda }\int_{B_{R}^1 }K\left( \left| x\right| \right) \left|
u\right|^{q} dx +\frac{1}{\Lambda }\int_{B_{R}^2}K\left( \left| x\right| \right) \left|
u\right|^{q} dx .$$
\noindent Applying H\"{o}lder inequality first with conjugate exponents $ \frac{1}{\beta }$ and $\frac{1}{1-\beta }$, 
then with $2^*$ and $\frac{2N}{N+2}$, we get
$$\frac{1}{\Lambda }\int_{B_{R}^1 }K\left( \left| x\right| \right) \left|
u\right| ^{q} dx \leq  \int_{B_{R}^1  }\left| x\right| ^{\alpha }V\left( \left| x\right|
\right) ^{\beta }\left| u\right| ^{q} dx =\int_{B_{R}^1  }\left| x\right| ^{\alpha }V\left( \left| x\right|
\right) ^{\beta }\left| u\right| ^{q-\beta } \left| u\right| ^{\beta } dx  $$
$$\leq \left( \int_{B_{R}^1}\left( \left| x\right| ^{\alpha }\left| u\right|
^{q-1}\left| u\right| ^{1-\beta }\right) ^{\frac{1}{1-\beta }}dx\right)
^{1-\beta }\left( \int_{B_{R}^1 }V\left( \left| x\right| \right) \left|
u\right| dx\right) ^{\beta } $$
$$\leq \frac{1}{C^{2\beta}_1}\left( \int_{B_{R}^1} \left| x\right| ^{\frac{\alpha}{1-\beta} }\left| u\right|
^{\frac{q-1}{1-\beta}}\left| u\right|  dx\right)
^{1-\beta }\left( \int_{B_{R}^1 }V\left( \left| x\right| \right) f(u)^2 dx\right) ^{\beta } $$
$$\leq \frac{1}{C^{2\beta}_1}\left( \int_{B_{R}^1}\left( \left| x\right| ^{\frac{\alpha}{1-\beta} }\left| u\right|
^{\frac{q-1}{1-\beta}} \right)^{\frac{2N}{N+2}} dx\right)
^{(1-\beta )\frac{N+2}{2N} } \left(  \int_{B_{R}^1 } |u|^{2^*}dx \right)^{\frac{1-\beta}{2^*}}  \left( \int_{B_{R}^1 }V\left( \left| x\right| \right) f(u)^2 dx\right) ^{\beta } $$
$$\leq \frac{S^{1-\beta}_N}{C^{2\beta}_1} \, m^{q-1} \left( \int_{B_{R}}\left| x\right| ^{\frac{\alpha -\nu (q-1)}{1-\beta}\frac{2N}{N+2} }  dx\right)
^{(1-\beta )\frac{N+2}{2N} }\, ||u||^{1-\beta} \, \left( \int_{B_{R} }V\left( \left| x\right| \right) f(u)^2 dx\right) ^{\beta } .$$
\noindent On the other hand, as $q>1>2\beta$, $\alpha > -N(1-\beta )$ and $|u(x)| < 1 $ in $B_{R}^2$, we get
$$\frac{1}{\Lambda }\int_{B_{R}^2 }K\left( \left| x\right| \right) \left|
u\right| ^{q} dx \leq  \int_{B_{R}^2  }\left| x\right| ^{\alpha }V\left( \left| x\right|
\right) ^{\beta }\left| u\right| ^{q-2\beta} |u|^{2\beta} dx \leq \left(\int_{B_{R}^2  }\left( \left| x\right|^{\alpha}  |u|^{q-2\beta} \right)^{\frac{1}{1-\beta} }
dx\right)^{1-\beta}
\left( \int_{B_{R}^2  } V \, \left| u\right| ^{2 }  dx \right)^{\beta} $$
$$\leq \frac{1}{C^{2\beta}_1} \left(\int_{B_{R}^2  } \left| x\right|^{\frac{\alpha}{1-\beta}}  
dx\right)^{1-\beta} \, \left( \int_{B_{R}^2  } V \, f(u)^2  dx \right)^{\beta} \leq C(N,\alpha , \beta) R^{N(1-\beta ) +\alpha} \, \left( \int_{B_{R}  } V \, f(u)^2  dx \right)^{\beta}.$$
\noindent The thesis follows by summing the two inequalities we have obtained.

\smallskip

\noindent \emph{Case }$\beta =\frac{1}{2}$.\ \ 
We have
$$\frac{1}{\Lambda }\int_{B_{R}^1}   K\left( \left| x\right| \right) \left|
u\right| ^{q}dx \leq \int_{B_{R}^1} \left| x\right|
^{\alpha }V\left( \left| x\right| \right) ^{1/2 }\left| u\right|
^{q} dx = \int_{B_{R}^1 }\left| x\right| ^{\alpha } \left| u\right| ^{q-{\frac{1}{2}}}V\left(
\left| x\right| \right) ^{\frac{1}{2}}\left| u\right|^{\frac{1}{2}} dx $$
$$ \leq \left( \int_{B_{R}^1 } \left| x\right| ^{2\alpha }  \left| u\right| ^{ 2q-1}dx\right)
^{\frac{1}{2}}    \left( \int_{B_{R}^1 }  V\left( \left| x\right| \right) \left|
u\right| dx\right) ^{\frac{1}{2}} $$
$$\leq \frac{m^{q-1/2}}{C_1}  \left( \int_{B_{R} } \left| x\right| ^{2\alpha -\nu (2q-1)} dx\right)^{\frac{1}{2}}  \left( \int_{B_{R}}  
V\left( \left| x\right| \right) f(u)^2dx\right) ^{\frac{1}{2}}, $$
\noindent while
$$\frac{1}{\Lambda }\int_{B_{R}^2}   K\left( \left| x\right| \right) \left|
u\right| ^{q}dx \leq \int_{B_{R}^2} \left| x\right|
^{\alpha }V\left( \left| x\right| \right) ^{1/2}  |u| \left| u\right|
^{q-1} dx $$
$$\leq  \left( \int_{B_{R}^2 } \left| x\right| ^{2\alpha}  |u|^{2(q-1) } dx \right)^{1/2}
\left( \int_{B_{R}^2 }  V\left( \left| x\right| \right) \left|
u\right|^2 dx\right) ^{1/2}  \leq \frac{1}{C_{1}}  \left( \int_{B_{R}} \left| x\right| ^{2\alpha } dx \right)^{1/2}
\left( \int_{B_{R}}  V\left( \left| x\right| \right) f(u)^2 dx\right)^{1/2}  $$
$$= C(N, \alpha , \beta ) R^{ \alpha +N/2 } \left( \int_{B_{R}}  V\left( \left| x\right| \right) f(u)^2 dx\right)^{1/2} .$$
\noindent As before, the thesis follows from the two inequalities we have obtained.

\noindent \emph{Case }$1/2<\beta <1$.\ \ 
\noindent We will apply H\"{o}lder inequality with conjugate exponents $p=p'={\frac{1}{2}}$, or $p=\frac{1}{2\beta -1}>1$ and $p'=\frac{1}{2-2\beta }$. As above,
we will estimate separately the two integrals $\int_{B_{R}^1}  K\left( | x| \right) \left|
u\right| ^{q}dx $ and $\int_{B_{R}^2}   K\left( | x| \right) \left| u\right| ^{q}dx $.  
We have                                                                         
$$\frac{1}{\Lambda }\int_{B_{R}^1}   K\left( \left| x\right| \right) \left|
u\right| ^{q}dx \leq \int_{B_{R}^1} \left| x\right|
^{\alpha }V\left( \left| x\right| \right) ^{\beta }\left| u\right|
^{q} dx = \int_{B_{R}^1 }\left| x\right| ^{\alpha }V\left(
\left| x\right| \right) ^{\frac{2\beta -1}{2}}\left| u\right| ^{q-{\frac{1}{2}}}V\left(
\left| x\right| \right) ^{\frac{1}{2}}\left| u\right|^{\frac{1}{2}} dx $$
$$ \leq \left( \int_{B_{R}^1 } \left| x\right| ^{2\alpha }V\left( \left|
x\right| \right) ^{2\beta -1}\left| u\right| ^{ 2q-1}dx\right)
^{\frac{1}{2}}    \left( \int_{B_{R}^1 }  V\left( \left| x\right| \right) \left|
u\right| dx\right) ^{\frac{1}{2}} $$
$$\leq \frac{1}{C_1}  \left( \int_{B_{R}^1 } \left| x\right| ^{2\alpha }  |u|^{2q -2\beta }V\left( \left|
x\right| \right) ^{2\beta -1}\left| u\right| ^{ 2\beta-1}dx\right)^{\frac{1}{2}}  \left( \int_{B_{R}^1 }  
V\left( \left| x\right| \right) f(u)^2dx\right) ^{\frac{1}{2}} $$
$$\leq \frac{1}{C_1}  \left( \int_{B_{R}^1 } \left| x\right| ^{\frac{\alpha}{1-\beta} }  |u|^{\frac{q-\beta}{1-\beta} } dx \right)^{1-\beta}
 \left( \int_{B_{R}^1 } V\left( \left|
x\right| \right) \left| u\right| dx\right)^{\frac{2\beta -1}{2}} \left( \int_{B_{R}^1 }  
V\left( \left| x\right| \right) f(u)^2dx\right) ^{\frac{1}{2}}  $$
$$\leq \frac{1}{C_{1}^{2\beta}} \, m^{q-\beta}  \left( \int_{B_{R} } \left| x\right| ^{\frac{\alpha -\nu (q-\beta)}{1-\beta} } dx \right)^{1-\beta}
\left( \int_{B_{R}}  
V\left( \left| x\right| \right) f(u)^2dx\right) ^{\beta}.$$ 

\noindent On the other hand 
$$\frac{1}{\Lambda }\int_{B_{R}^2}   K\left( \left| x\right| \right) \left|
u\right| ^{q}dx \leq \int_{B_{R}^2} \left| x\right|
^{\alpha }V \left( \left| x\right| \right) ^{\beta }\left| u\right|
^{q} dx = \int_{B_{R}^2 }\left| x\right| ^{\alpha }V \left(
\left| x\right| \right) ^{\beta}\left| u\right| ^{2\beta } \left| u\right|^{q-2\beta} dx $$
$$\leq  \left( \int_{B_{R}^2 } \left| x\right| ^{\frac{\alpha}{1-\beta} }  |u|^{\frac{q-2\beta}{1-\beta} } dx \right)^{1-\beta}
\left( \int_{B_{R}^2 }  V\left( \left| x\right| \right) \left|
u\right|^2 dx\right) ^{\beta}  \leq \frac{1}{C_{1}^{2\beta}}  \left( \int_{B_{R}} \left| x\right| ^{\frac{\alpha}{1-\beta} } dx \right)^{1-\beta}
\left( \int_{B_{R}}  V\left( \left| x\right| \right) f(u)^2 dx\right) ^{\beta}  $$
$$= C(N, \alpha , \beta ) R^{N(1-\beta ) + \alpha } \left( \int_{B_{R}}  V\left( \left| x\right| \right) f(u)^2 dx\right) ^{\beta} .$$

\noindent As before, the thesis follows from the two inequalities that we have obtained.

\noindent \emph{Case }$\beta =1$.\ \ 
\noindent Recall that $\beta =1$ implies $q>2$ and $\alpha >0$. We have 
$$\frac{1}{\Lambda }\int_{B_{R}^1  }  K\left( \left| x\right| \right) \left|
u\right| ^{q} dx \leq \int_{B_{R}^1  } |x|^{\alpha} V\left( \left| x\right| \right)^{1/2} |u|^{q-1/2} V\left( \left| x\right| \right)^{1/2} |u|^{1/2} dx$$

$$\leq \left( \int_{B_{R}^1  } |x|^{2\alpha} V\left( \left| x\right| \right)  |u|^{2q-1} dx \right)^{1/2} \left( \int_{B_{R}^1  }  V\left( \left| x\right| \right)  |u| dx \right)^{1/2} $$
$$\leq
\frac{1}{C_1} \left( \int_{B_{R}^1  } |x|^{2\alpha} |u|^{2q-2} V\left( \left| x\right| \right)  |u| dx \right)^{1/2}
 \left( \int_{B_{R}^1  }  V\left( \left| x\right| \right)  f(u)^2 dx \right)^{1/2}  $$

$$ \leq \frac{m^{q-1}}{C^{2}_1} \left( \int_{B_{R}  } |x|^{2\alpha -2 \nu (q-1)} V\left( \left| x\right| \right) f(u)^2 dx \right)^{1/2}
 \left( \int_{B_{R}  }  V\left( \left| x\right| \right)  f(u)^2 dx \right)^{1/2} .$$

\noindent On the other hand
$$\frac{1}{\Lambda }\int_{B_{R}^2  }  K\left( \left| x\right| \right) \left|
u\right| ^{q} dx \leq \int_{B_{R}^2  } |x|^{\alpha}  |u|^{q-2} V\left( \left| x\right| \right)  |u|^2 dx \leq \frac{1}{C^{2}_1} \, R^{\alpha} \, \int_{B_{R}  }  V\left( \left| x\right| \right)  f(u)^2 dx.$$

\noindent The thesis easily follows.
\end{proof}

	The following lemma is analogous to the previous one, dealing with $B_R^{c}$ instead of $B_R$.

\begin{lem}
\label{Lem(Omega2)}Let $R_0>0$ and assume that
\[
\Lambda :=\sup_{x\in B_{R_0}^c }\frac{K\left( \left| x\right|
\right) }{\left| x\right| ^{\alpha }V\left( \left| x\right| \right) ^{\beta }%
}<+\infty \quad \text{for some }0\leq \beta \leq 1\text{~and }\alpha \in 
\mathbb{R}.
\]
Let $u\in E$ and assume that there exist $\nu , m>0$
such that 
\[
\left| u\left( x\right) \right| \leq \frac{m}{\left| x\right| ^{\nu }}\quad 
\text{almost everywhere in } B_{R_0^c} .
\]
Set $ \gamma (m) := m/R_0^{\nu}+1$. Then there exists a constant $C=C(N, R_0,\beta )>0$ such that $\forall R > R_{0}$ and $\forall q>\max \left\{ 1,2\beta \right\} $, one has 
\medskip

$\displaystyle\int_{B_R^c }K\left( \left| x\right| \right) \left| u\right| ^{q-1}\left|
h\right| dx$
\[
\leq \left\{ 
\begin{array}{ll}
\Lambda  C \gamma (m)^{2 \beta} m^{q-1} \left( \int_{B^c_R} \left| x\right| ^{\frac{
\alpha -\nu \left( q-1\right) }{N+2\left( 1-2\beta  \right) }2N}dx\right)^{
\frac{N+2\left( 1-2\beta \right) }{2N}}\left\| u\right\|^{1-2\beta } \left(\int_{B^c_R} V f(u)^2 dx  \right)^{\beta}    \quad \medskip  & 
\text{if }0\leq \beta \leq \frac{1}{2} \\ 
\Lambda  C \gamma (m)^{2 \beta} m^{q-2\beta}  \left( \int_{B^c_R }\left| x\right| ^{\frac{\alpha
-\nu \left( q-2\beta \right) }{1-\beta }}dx\right) ^{1-\beta }  \left(\int_{B^c_R} V f(u)^2 dx  \right)^{\beta} \medskip  & \text{if }\frac{1}{2}
<\beta <1 \\ 
\Lambda  C \gamma (m)^{2 } m^{q-2}  \left( \int_{B^c_R }\left| x\right| ^{2(\alpha -\nu \left(
q-2\right)) }V\left( \left| x\right| \right) f(u)^2 dx\right) ^{
\frac{1}{2}}  \left(\int_{B^c_R} V f(u)^2 dx  \right)^{\frac{1}{2}} & \text{if }\beta =1.
\end{array}
\right. 
\]
\end{lem}

\begin{proof}

We start by noticing that, thanks to the hypotheses, we have 
$$|u(x)| \leq \frac{m}{|x|^{\nu}}  \leq \frac{m}{R_0^{\nu}} \quad\text{for all }|x| \geq R_0.$$
Since $\gamma (m)= m/R_0^{\nu} +1$, we have $|u(x)| \leq \gamma (m)$ in $B_{R_0}^c$ and $\gamma (m) \geq 1$. Recalling that $|f(t)| \geq C_1 |t|$ when $|t|\leq 1$, and that $f(t)^2$ is even and increasing on $\mathbb{R}_{+}$, for all $R\geq R_0$ we have
$$ \int_{B_R^c} V\left( \left| x\right| \right)  |u|^2 dx = \gamma(m)^2 \int_{B_R^c} V\left( \left| x\right| \right)  \left|\frac{u}{\gamma (m)}\right|^2 dx \leq \left( \frac{\gamma (m)}{C_1} \right)^2 \int_{B_R^c} 
V\left( \left| x\right| \right)  f \left(\left|\frac{u}{\gamma (m)}\right| \right)^2 dx $$
$$\leq  \left( \frac{\gamma (m)}{C_1} \right)^2 \int_{B_R^c} V\left( \left| x\right| \right)  f \left( u \right)^2 dx .$$

\noindent \emph{Case $\beta =0$.}\ \ 
\noindent Here the argument is exactly the same as in the case $\beta =0$ of the previous lemma, so we do not repeat it. We apply H\"older inequality with exponents $2^* = \frac{2N}{N-2}$ and $\frac{2N}{N+2}$, together with the standard Sobolev inequality, to get
$$
\frac{1}{\Lambda }\int_{B_R^c }K\left( \left| x\right| \right) \left|
u\right| ^{q} dx \leq m^{q-1} C \left( \int_{B_R^c }\left| x\right| ^{\frac{\alpha -\nu
\left( q-1\right) }{N+2}2N}dx\right) ^{\frac{N+2}{2N}}\left\| u\right\| .
$$

\noindent \emph{Case $0<\beta < \frac{1}{2}$. }\ \ 
Thanks to H\"older inequalities with pairs of conjugate exponents $\frac{1}{\beta}$ and $\frac{1}{1-\beta}$, and $\frac{2^* (1-\beta )}{1-2\beta}$ and 
$\frac{2N (1-\beta )}{N+2(1-2\beta )}$, we obtain
$$
\frac{1}{\Lambda }\int_{B_R^c } K\left( \left| x\right| \right) \left|
u\right| ^{q} dx \leq \int_{B_R^c } |x|^{\alpha } |u|^{q-2\beta} V\left( \left| x\right| \right)^{\beta} |u|^{2 \beta} dx$$
$$\leq \left( \int_{B_R^c }  \left( |x|^{\alpha } |u|^{q-2\beta} \right)^{\frac{1}{1-\beta}}  dx  \right)^{1-\beta}   \left( \int_{B_R^c } 
V\left( \left| x\right| \right)  |u|^2 dx \right)^{\beta} $$
$$\leq \left( \frac{\gamma (m)}{C_1} \right)^{2 \beta} \left( \int_{B_R^c } \left( |x|^{\alpha }|u|^{q-1} |u|^{1-2\beta} \right)^{\frac{1}{1-\beta}}dx\right)^{1-\beta} 
 \left( \int_{B_R^c } 
V f(u)^2 dx \right)^{\beta}$$
$$\leq \left( \frac{\gamma (m)}{C_1} \right)^{2 \beta} \left( \int_{B_R^c } \left( |x|^{\frac{\alpha}{1-\beta} }|u|^{\frac{q-1}{1-\beta}} 
\right)^{\frac{2N(1-\beta )}{N+2(1-2\beta )}}dx\right)^{\frac{N+2(1-2\beta )}{2N}} \left(\int_{B_R^c }| u|^{2^*}dx \right)^{\frac{1-2\beta}{2^*}}
\left( \int_{B_R^c } V f(u)^2 dx \right)^{\beta}$$
$$\leq \left( \frac{\gamma (m)}{C_1} \right)^{2 \beta}   C_N^{1-\beta} m^{q-1}\left( \int_{B_R^c }  |x|^{\frac{\alpha -\nu (q-1)}{N+2 (1-2\beta )}2N }
dx\right)^{\frac{N+2(1-2\beta )}{2N}} ||u||^{1-2\beta}
\left( \int_{B_R^c } V f(u)^2 dx \right)^{\beta}.
$$
The result follows with $C= C_N^{1-\beta}/C_1^{2 \beta} $.

\noindent \emph{Case $\beta = \frac{1}{2}$. }\ \ 
We have
$$
\frac{1}{\Lambda }\int_{B_R^c } K\left( \left| x\right| \right) \left|
u\right| ^{q} dx \leq \int_{B_R^c } |x|^{\alpha } |u|^{q-1} V\left( \left| x\right| \right)^{1/2} |u|dx$$
$$\leq \left( \int_{B_R^c }   |x|^{2\alpha } |u|^{2(q-1)}  dx  \right)^{1/2}   \left( \int_{B_R^c } 
V |u|^2 dx \right)^{1/2} $$
$$\leq \frac{\gamma (m)}{C_1}  \, m^{q-1}\left( \int_{B_R^c }  |x|^{2\alpha - 2 \nu (q-1)} dx \right)^{1/2} \left( \int_{B_R^c } 
V f(u)^2 dx \right)^{1/2}.$$

\noindent \emph{Case $ \frac{1}{2}< \beta < 1$. }\ \ 
\noindent We use H\"older inequality with exponents $p=p'= \frac{1}{2}$ first, and then with $p=\frac{1}{2\beta -1}$ and $p' =  \frac{1}{2 -2\beta}$. We get
$$
\frac{1}{\Lambda }\int_{B_R^c } K\left( \left| x\right| \right) \left|
u\right| ^{q} dx \leq \int_{B_R^c } |x|^{\alpha } |u|^{q} V\left( \left| x\right| \right)^{\beta} dx =
\int_{B_R^c }   |x|^{\alpha } V\left( \left| x\right| \right)^{\beta -\frac{1}{2}}|u|^{q-1} V\left( \left| x\right| \right)^{\frac{1}{2}} |u| dx $$
$$\leq \left( \int_{B_R^c }   |x|^{2\alpha } V\left( \left| x\right| \right)^{2\beta -1}|u|^{2q-2} dx  \right)^{1/2}   \left( \int_{B_R^c } 
V |u|^2 dx \right)^{1/2} $$
$$\leq \frac{\gamma (m)}{C_1} \left( \int_{B_R^c }   |x|^{2\alpha } |u|^{2(q-2\beta )}V\left( \left| x\right| \right)^{2\beta -1}|u|^{2(2\beta -1)} dx  \right)^{1/2}   \left( \int_{B_R^c } 
V f(u)^2 dx \right)^{1/2}$$
$$\leq \frac{\gamma (m)}{C_1} \left( \int_{B_R^c }   |x|^{\frac{\alpha}{1-\beta} } |u|^{\frac{q-2\beta}{1-\beta} } dx  \right)^{1-\beta} 
  \left( \int_{B_R^c } 
V\left( \left| x\right| \right) |u|^2 dx \right)^{\frac{2\beta -1}{2}}  \left( \int_{B_R^c } 
V f(u)^2 dx \right)^{1/2}$$
$$\leq \left( \frac{\gamma (m)}{C_1} \right)^{2\beta} \, m^{q-2\beta} \left( \int_{B_R^c }   |x|^{\frac{\alpha -\nu (q-2\beta )}{1-\beta} } dx  \right)^{1-\beta}\left( \int_{B_R^c } V f(u)^2 dx \right)^{\beta}.
$$

\noindent \emph{Case $ \beta = 1$. }\ \ 
\noindent In this case, hypothesis $q > \max \{ 1, 2 \beta \}$ implies $q>2$. Hence
$$
\frac{1}{\Lambda }\int_{B_R^c } K\left( \left| x\right| \right) \left|
u\right| ^{q} dx \leq \int_{B_R^c } |x|^{\alpha } |u|^{q} V\left( \left| x\right| \right) dx =
\int_{B_R^c }   |x|^{\alpha } V\left( \left| x\right| \right)^{\frac{1}{2}}|u|^{q-1} V\left( \left| x\right| \right)^{\frac{1}{2}} |u| dx $$
$$\leq \left( \int_{B_R^c }   |x|^{2\alpha } V\left( \left| x\right| \right) |u|^{2q-2} dx  \right)^{1/2}   \left( \int_{B_R^c } 
V\left( \left| x\right| \right) |u|^2 dx \right)^{1/2} $$
$$\leq \frac{\gamma (m)}{C_1} \left( \int_{B_R^c }   |x|^{2\alpha } |u|^{2(q-2 )}V\left( \left| x\right| \right)|u|^{2} dx  \right)^{1/2}   \left( \int_{B_R^c } 
V f(u)^2 dx \right)^{1/2}$$
$$\leq \left( \frac{\gamma (m)}{C_1} \right)^2 \, m^{q-2}\left( \int_{B_R^c }   |x|^{2\alpha -2\nu (q-2)} V f(u)^2 dx  \right)^{1/2}   \left( \int_{B_R^c } 
V f(u)^2 dx \right)^{1/2}. 
$$

\end{proof}

	We can now prove Theorems \ref{THM0} and \ref{THM1}.

\begin{proof}[Proof of Theorem \ref{THM0}]

Assume the hypotheses of the theorem and let $u\in E$ be
such that $\left\| u\right\| =1$. Let $0<R< R_{1}$. We
will denote by $C$ any positive constant which does not depend on $u$ and $R$.
Recalling the pointwise estimates of Corollary \ref{COR:est} and the fact that 
\[
\sup_{x\in B_{R}}\frac{K\left( \left| x\right| \right) }{\left|
x\right| ^{\alpha _{0}}V\left( \left| x\right| \right) ^{\beta _{0}}}\leq 
\sup_{r\in \left( 0,R_{1}\right) }\frac{K\left( r\right) }{%
r^{\alpha _{0}}V\left( r\right) ^{\beta _{0}}}<+\infty , 
\]
we can apply Lemma \ref{Lem(Omega)} with $R_0=R_1$, $\alpha =\alpha
_{0}$, $\beta =\beta _{0}$, $m=M\left\| u\right\| =M$ and $\nu =
\frac{N-2}{2}$. 
The argument will proceed as follows: we will distinguish several cases, as in Lemma \ref{Lem(Omega)}, and we will prove that in any case we get
\begin{equation}
\int_{B_{R}}K\left( \left| x\right| \right) \left| u\right| ^{q_{1}} dx \leq C R^{\delta}
\quad\text{for any }0<R<R_1,
\end{equation}
with $\delta >0$ and $C>0$ independent from $R$ and $u$. This clearly implies $\mathcal{S}_{0}\left( q_{1},R \right) \leq C R^{\delta}$, and hence $\lim_{R\rightarrow 0^+}\mathcal{S}_{0}\left( q_{1},R \right) =0$. Recall also that if $||u||=1$ then $\int_{\mathbb{R}^{N}} V(|x|) f(u)^2 dx \leq C$, for a suitable $C>0$ independent from $u$. 

If $\beta _{0}=0$, we get 
$$
\int_{B_{R}} K\left( \left| x\right| \right) \left| u\right|^{q_{1}} dx \leq C\left( \int_{B_{R}}| x|^{\frac{\alpha _{0}-
\nu\left( q_{1}-1\right) }{N+2 }2N }  dx\right)^{\frac{N+2 }{2N}} \leq  C\left( \int_{0}^R \rho^{\frac{\alpha _{0}-
\frac{N-2}{2}\left( q_{1}-1\right) }{N+2}2N  +N -1}
 d\rho\right)^{\frac{N+2}{2N}} ,
$$
where
$$\frac{\alpha _{0}-
\frac{N-2}{2}\left( q_{1}-1\right) }{N+2 }
2N  +N = \frac{N}{N+2} \left[ \, 2\alpha_0 - (N-2) (q_1 -1) +N+2   \,  \right] =  $$  
$$= \frac{N \left( N-2 \right)}{N+2} \left[  \frac{2\alpha_0 +2N}{N-2} - q_1    \right] =  \frac{N(N-2)}{N+2}  \left[ \, q_{0}^* (\alpha_0 , 0 )-q_1 \, \right] >0,$$
 thanks to the hypotheses. Hence, by integration and simple computations, we deduce 
$$\int_{B_{R}}K\left( \left| x\right| \right) \left| u\right| ^{q_{1}} dx 
\leq C R^{\frac{N-2}{2}\left[ q_{0}^*  ( \alpha_0 , 0 )-q_1 \right]}=C R^{\delta}.
$$

If $0<\beta _{0}<1/2$, we have 
$$
\int_{B_{R}} K\left( \left| x\right| \right) \left| u\right|^{q_{1}} dx \leq C \left[ \left( \int_{B_{R}}| x|^{\frac{\alpha _{0}-
\frac{N-2}{2}\left( q_{1}-1\right) }{(N+2) (1-\beta_0 ) }2N }  dx  \right)^{\frac{(N+2 )(1-\beta_0 ) }{2N}}  + R^{\alpha_0 +N (1-\beta_0 )} \right] ,
$$
where
$$
 \int_{B_{R}}| x|^{\frac{\alpha _{0}-
\frac{N-2}{2} \left( q_{1}-1\right) }{(N+2) (1-\beta_0 ) }2N }  dx = C \int_{0}^R \rho^{\frac{\alpha _{0}-
\frac{N-2}{2}\left( q_{1}-1\right) }{(N+2) (1-\beta_0 ) }2N +N-1}  d\rho .
$$
Now observe that
$$\frac{\alpha_0 -\frac{N-2}{2} (q_1 - 1)}{(N+2) (1-\beta_0 )}2N +N = \frac{N}{(N+2) (1-\beta_0 )} \left( 2\alpha_0 +2N -\beta_0 (N+2) - (N-2)q_1 \right)
$$
$$
=\frac{N (N-2)}{(N+2) (1-\beta_0 )} \left( \frac{2\alpha_0 +2N -\beta_0 (N+2) }{N-2} - q_1 \right) 
=\frac{N (N-2)}{(N+2) (1-\beta_0 )} \left( q_{0}^*  (\alpha_0 , \beta_0 ) - q_1 \right)  >0 ,
$$
so that
$$
 \int_{B_{R}}| x|^{\frac{\alpha _{0}-
\frac{N-2}{2} \left( q_{1}-1\right) }{(N+2) (1-\beta_0 ) }2N }  dx = C R^{\frac{N-2}{2} \left( q_{0}^*  (\alpha_0 , \beta_0 ) - q_1 \right) }.
$$
On the other hand, one has $\alpha_0 + N (1-\beta_0 ) >0$ by hypothesis. Hence as $R\rightarrow 0^+$ we have
$$
\mathcal{S}_{0}\left( q_{1},R \right) \leq C R^{\frac{N-2}{2} \left( q_{0}^*  (\alpha_0 , \beta_0 ) - q_1 \right) } + C R^{\alpha_0 + N (1-\beta_0 ) } \leq
C R^{\delta},
$$
where $\delta = \min \left\{  \frac{N-2}{2} \left( q_{0}^*  (\alpha_0 , \beta_0 ) - q_1 \right)  , \alpha_0 + N (1-\beta_0 )  \right\}>0 $.

If $\beta _{0}=1/2$, we have 
$$
\int_{B_{R}} K\left( \left| x\right| \right) \left| u\right|^{q_{1}} dx \leq C \left[ \left( \int_{B_{R}}| x|^{2\alpha _{0}-
\frac{N-2}{2}\left(2 q_{1}-1\right) }  dx  \right)^{\frac{1}{2}}  + R^{\alpha_0 + \frac{N}{2} } \right] ,
$$
where
$$
 \int_{B_{R}}| x|^{2\alpha _{0}-
\frac{N-2}{2}\left(2 q_{1}-1\right) }  dx = C \int_{0}^R  \rho^{2\alpha _{0}-
\frac{N-2}{2}\left(2 q_{1}-1\right) +N -1}d\rho 
$$                                                                                          
and 
$$
2\alpha _{0}-
\frac{N-2}{2}\left(2 q_{1}-1\right) +N = 2\alpha _{0}+
\frac{3}{2}N  -1 - (N-2)q_{1} = (N-2) \left( \frac{2\alpha _{0}+
\frac{3}{2}N  -1 }{N-2} -q_1 \right)$$
$$= (N-2)\left( q_{0}^*  \left( \alpha_0 , \frac{1}{2}\right) -q_1 \right)  >0.
$$
Hence we get
$$
 \int_{B_{R}}| x|^{2\alpha _{0}-
\frac{N-2}{2}\left(2 q_{1}-1\right) }  dx= C R^{\frac{N-2}{2}  \left( q_{0}^*  \left( \alpha_0 , 1/2  \right) -q_1 \right)  },
$$
and, recalling that $\alpha_0 + \frac{N}{2}>0 $ by hypothesis, for $R\rightarrow 0^+$ we have
$$
\mathcal{S}_{0}\left( q_{1},R \right) \leq C R^{\frac{N-2}{2} \left( q_{0}^*  (\alpha_0 , 1/2) - q_1 \right) } + C R^{\alpha_0 + \frac{N}{2} } \leq
C R^{\delta}
$$
\noindent with
$\delta = \min \left\{  \frac{N-2}{2} \left( q_{0}^*  \left( \alpha_0 , \frac{1}{2} \right) - q_1 \right)  , \alpha_0 + \frac{N}{2}   \right\}>0 $.

 If $1/2 <\beta _{0}<1 $, we have 
$$
\int_{B_{R}} K\left( \left| x\right| \right) \left| u\right|^{q_{1}} dx \leq C \left[ \left( \int_{B_{R}}| x|^{\frac{\alpha_0 - \frac{N-2}{2}(q_1 -\beta_0 )  }{1-\beta_0}}  dx  \right)^{1-\beta_0}  + R^{\alpha_0 + N(1-\beta_0 )} \right] .
$$
where
$$
 \int_{B_{R}}| x|^{\frac{\alpha_0 - \frac{N-2}{2}(q_1 -\beta_0 )  }{1-\beta_0} }  dx  = 
C \int_{0}^R \rho^{ \frac{\alpha_0 - \frac{N-2}{2}(q_1 -\beta_0 )  }{1-\beta_0} +N } d\rho 
$$
and
$$
\frac{\alpha_0 - \frac{N-2}{2}(q_1 -\beta_0 )  }{1-\beta_0} +N = \frac{1}{2(1-\beta_0 )}  \left( 2 \alpha_0 +2N -\beta_0 (N+2) - (N-2)q_1   \right)$$
$$=
\frac{N-2}{2(1-\beta_0 )} \left( \frac{2 \alpha_0 +2N -\beta_0 (N+2) }{N-2} -q_1 \right) = \frac{N-2}{2(1-\beta_0 )} 
\left( q_{0}^*  (\alpha_0 , \beta_0 ) -q_1 \right)  >0 .
$$
So
$$
 \left( \int_{B_{R}}| x|^{\frac{\alpha_0 - \frac{N-2}{2}(q_1 -\beta_0 )  }{1-\beta_0} }  dx\right)^{1-\beta_0} = 
C R^{\frac{N-2}{2} 
\left(q_{0}^*  (\alpha_0 , \beta_0 ) -q_1 \right) }.
$$
Then, as $R\rightarrow 0^+$, we obtain
$$
\mathcal{S}_{0}\left( q_{1},R \right) \leq C R^{\frac{N-2}{2} \left(q_{0}^*  (\alpha_0 , \beta_0 ) - q_1 \right) } + C R^{\alpha_0 + N(1-\beta_0 ) } \leq
C R^{\delta}
$$
with
$\delta = \min \left\{  \frac{N-2}{2} \left(q_{0}^*  \left( \alpha_0 , \beta_0 \right) - q_1 \right)  , \alpha_0 + N(1-\beta_0 )   \right\}>0 $.

 If $\beta _{0}=1 $, then we have 
$$
\int_{B_{R}} K\left( \left| x\right| \right) \left| u\right|^{q_{1}} dx \leq C \left[ \left( \int_{B_{R}}| x|^{2\alpha_0 - (N-2)(q_1 -1 ) }  V f(u)^2 dx  
\right)^{\frac{1}{2}}  + R^{\alpha_0 } \right] .
$$
\noindent Notice that $\alpha_0 > -N(1-\beta_0 )$ means $\alpha_0 >0$, since $\beta_0 =1$. Notice also that 
$$
2\alpha_0 - (N-2)(q_1 -1 )= (N-2) \left( \frac{2\alpha_0 +N -2}{N-2} -q_1 \right) = (N-2 ) \left( q_{0}^*  (\alpha_0 ,1 ) -q_1 \right)>0 
$$
implies
$$
\left( \int_{B_{R}} | x|^{2\alpha_0 - (N-2)(q_1 -1 ) }  V f(u)^2 dx  
\right)^{\frac{1}{2}} \leq R^{\alpha_0 - \frac{N-2}{2}(q_1 -1 )}  \left( \int_{B_{R}}   V f(u)^2 dx  
\right)^{\frac{1}{2}}\leq C R^{\alpha_0 - \frac{N-2}{2}(q_1 -1 )} .
$$
Hence, as $R\rightarrow 0^+ $, we get 
$$
\mathcal{S}_{0}\left( q_{1},R \right) \leq C R^{\alpha_0 - \frac{N-2}{2}(q_1 -1 )} + C R^{\alpha_0  } \leq
C R^{\alpha_0 - \frac{N-2}{2}(q_1 -1 )}
$$
with ${\alpha_0 - \frac{N-2}{2}(q_1 -1 )} =\delta >0$.

 As a conclusion, in any case, we have $\mathcal{S}_{0}\left( q_{1},R\right) \leq
CR^{\delta }$ for some $\delta =\delta \left( N,\alpha _{0},\beta_{0},q_{1}\right) >0$ and the proof is thus complete.
\end{proof}

\proof[Proof of Theorem \ref{THM1}]
Assume the hypotheses of the theorem and let $u \in E$ be
such that $\left\| u\right\| =1$. Let $R> R_{2}$. We
will denote by $C$ any positive constant which does not depend on $u$ and $R$. We will separate three different cases and we will get, in each one, an 
inequality of the following form:
$$
\int_{B_{R}^{c}}K\left( \left| x\right| \right) \left| u\right|
^{q_{2}} dx \leq C R^{\delta}
$$
with $C>0$ and $\delta <0$ independent from $R,u$. This clearly gives $\mathcal{S}_{\infty}\left( q_{2},R \right) \leq  C R^{\delta}$, and hence
$\lim_{r\rightarrow +\infty}  \mathcal{S}_{\infty}\left( q_{2},R \right) =0$.
As in the proof of the previous theorem, by pointwise estimates and the fact that 
\[
\sup_{x\in B_{R}^{c}}\frac{K\left( \left| x\right| \right) }{%
\left| x\right| ^{\alpha _{\infty }}V\left( \left| x\right| \right) ^{\beta
_{\infty }}}\leq \sup_{r>R_{2}}\frac{K\left( r\right) }{%
r^{\alpha _{\infty }}V\left( r\right) ^{\beta _{\infty }}}<+\infty , 
\]
we can apply Lemma \ref{Lem(Omega2)} with $R_0= R_2$, $\alpha
=\alpha _{\infty }$, $\beta =\beta _{\infty }$, $m=M \left\|u\right\|
=M $ and $\nu =\frac{N-2}{2}$. Recall also that $||u||=1$ inplies $\int_{\mathbb{R}^{N}}V f(u)^2 dx \leq C$ with $C$ independent from $u$. The computations of the present proof are essentially the same of those in the proof of Theorem 3 in \cite{BGR_I}: the function there called $q^*$ is the same as the function $q_{\infty}^*$ here. Hence, we will be a little sketchy here.

If $0\leq \beta _{\infty }\leq 1/2$, we get 
$$
\int_{B_{R}^{c}}K\left( \left| x\right| \right) \left| u\right|
^{q_{2}} dx \leq C\left( \int_{B_{R}^{c}}\left| x\right|
^{\frac{\alpha _{\infty }- \frac{N-2}{2} \left( q_{2}-1\right) } {N+2(1-2\beta_{\infty}) }2N}dx\right) ^{\frac{N+2(1-2\beta_{\infty}) }{2N}} 
$$
$$
=C \left( R^{\frac{2\alpha_{\infty }-4\beta_{\infty }    +2N  -(N-2)q_2}{N+2(1-2\beta_{\infty } )}N}\right)^{\frac{N+2(1-2\beta_{\infty}) }{2N}},
$$
since $2\alpha_{\infty }-4\beta_{\infty }    +2N  -(N-2)q_2=   (N-2) \left( q_{\infty}^* (\alpha_{\infty } , \beta_{\infty }  ) -q_2 \right)<0.$

On the other hand, if $1/2<\beta _{\infty }<1$, then we have 
$$
\int_{B_{R}^{c}}K\left( \left| x\right| \right) \left| u\right|
^{q_{2}} dx \leq C\left( \int_{B_{R}^{c}}\left| x\right|
^{\frac{\alpha _{\infty }-\frac{N-2}{2} \left( q_{2}-2\beta _{\infty }\right) 
}{1-\beta _{\infty }}}dx\right) ^{1-\beta _{\infty }} =
C \left( R^{ \frac{2\alpha_{\infty } -(N-2)(q_2 -2\beta_{\infty })}{2(1-\beta_{\infty })}    }  \right)^{1-\beta _{\infty }},
$$
since
$$
\frac{2\alpha_{\infty } -(N-2)(q_2 -2\beta_{\infty })}{2(1-\beta_{\infty })} =\frac{N-2}{2(1- \beta_{\infty })}\left( q_{\infty}^* (\alpha_{\infty } , \beta_{\infty }) -q_2 \right) <0.
$$                                                                                  
                                                                                   
Finally, if $\beta _{\infty }=1$, we obtain 
$$
\int_{B_{R}^{c}}K\left( \left| x\right| \right) \left| u\right|^{q_{2}} dx 
\leq C\left( \int_{B_{R}^{c}}  \left| x\right|^{2\alpha _{\infty }- (N-2)\left( q_{2}-2\right)  }V\left( \left|x\right| \right) f(u)^2dx \right)^{\frac{1}{2}} ,
$$
$$
\leq C \left( R^{2\alpha _{\infty }- (N-2)\left( q_{2}-2\right) }\int_{B_{R}^{c}} V\left( \left|x\right| \right) f(u)^2dx \right)^{\frac{1}{2}}\leq
C R^{\frac{2\alpha _{\infty }- (N-2)\left( q_{2}-2\right) }{2}},
$$
since
$$2\alpha _{\infty }- (N-2)\left( q_{2}-2\right)  = (N-2)\left( q_{\infty}^* (\alpha_{\infty}, \beta_{\infty})-q_2 \right) <0 . $$

So, in any case, we get $\mathcal{S}_{\infty }\left(
q_{2},R\right) \leq CR^{\delta }$ for some $\delta =\delta ( N,p,\alpha_{\infty },\beta _{\infty },q_{2}) <0$, 
anf this completes the proof.
\endproof


\section{Critical points in the Orlicz-Sobolev space \label{SEC: dual}}

In this section we study the relations between the equation (\ref{EQdual}) and the original equation, that is
\begin{equation}
-\Delta w+ V\left( \left| x\right| \right) w - w \left( \Delta w^2 \right)= K(|x|) g(w) \quad \text{in }\mathbb{R}^{N}  \label{EQg}
\end{equation}
where $g$ satifies the assumptions stated in Section \ref{SEC: hp}. 
For both the equations, the solutions we get must be understood in two ways, weak  and classical (in $\mathbb{R}^{N} \backslash \{ 0 \} $). As to 
(\ref{EQg}) we will get weak solutions, that is, functions $w \in X$ satisfying, for all 
$ h\in C_{\mathrm{c}, r}^{\infty }( \mathbb{R}^{N} )$,

\begin{equation}
\int_{\mathbb{R}^{N}}\left( 1+ 2w^2 \right) \nabla w\cdot \nabla h\, dx+ \int_{ \mathbb{R}^{N} }   2w |\nabla w|^2 \, h \, dx + \int_{
\mathbb{R}^{N}}V\left( \left| x\right| \right) wh\,dx=\int_{
\mathbb{R}^{N}}K(|x|) g(w)  h\,dx ,  
\label{EQweak}
\end{equation}

\noindent which is obviously a weak formulation of (\ref{EQg}). We also prove that the solutions that we get are in $C^{2}(\mathbb{R}^{N} \backslash \{ 0 \} )$ and are classical solutions of (\ref{EQg}) in $\mathbb{R}^{N} \backslash \{ 0 \} $.

%
%
%
%
%
%
%
%

As we said in the introduction, we will obtain solutions by variational techniques, studying a functional related to the original 
problem by a change of variable. Let us define $I:E\rightarrow \mathbb{R}$ by setting

\begin{equation}\label{funct}
I(u) := \frac{1}{2}\int_{\mathbb{R}^{N}}|\nabla u |^2 dx + \frac{1}{2}\int_{\mathbb{R}^{N}}V(|x|) f(u)^2  dx-\int_{\mathbb{R}^{N}}K(|x|) G(f (u)) \,dx .
\end{equation}

\noindent In the following theorem, we state the main properties of $I$.

\begin{thm} 
\label{THM:critical}
Assume $N\geq 3$ and hypothesis $\left( \mathbf{H}\right) $. Assume that $g: {\mathbb{R}} \rightarrow {\mathbb{R}}$ is a continuous function satisfying ${\bf \left( g_{1}\right) } $, ${\bf \left( g_{2}\right) } $ and $  \left( {\bf g}_{q_{1},q_{2}}\right) $ with $q_1 , q_2$ satisfying $\left( {\cal S}_{q_{1},q_{2}}^{\prime }\right) $ (see Section \ref{COMP}). Then we have:
\begin{itemize}
\item[$\bullet$] $I$ is well defined and continuous in $E$.

\item[$\bullet$] $I$ is a $C^1$ map on $E$ and, for any $u \in E$, its differential $I' (u)$ is given by
\begin{equation} 
\label{diffrechet}
I' (u) h = \int_{\mathbb{R}^{N}}\nabla u \nabla h dx + \int_{\mathbb{R}^{N}}V(|x|) f (u) f'(u) h  dx -\int_{\mathbb{R}^{N}}K(|x|) g(f (u)) f'(u) h\,dx 
\end{equation}
for all $h \in E$.

\end{itemize}

\end{thm}

\begin{proof}

Let us define
$$ I_1 (u)  =    \frac{1}{2} \int_{\mathbb{R}^{N}} |\nabla u |^2 dx, \quad I_2 (u)  = \frac{1}{2} \int_{\mathbb{R}^{N}}V(|x|) f(u)^2  dx , \quad I_3 (u) = 
\int_{\mathbb{R}^{N}}K(|x|) G(f (u))  \,dx .$$
and study these three functionals.
\par \noindent As to $I_1$, it is a standard task to get that $I_1$ is $C^1$ in $E$ with differential given by $I'_1 (u)h= \int_{\mathbb{R}^{N}}\nabla u \nabla h dx $. 
\par \noindent As to $I_3$, we notice that, setting $h(x,t)= K(|x|) G(f(t))$, we have $h(x,t) = \int_{0}^{t} K(|x|) g(f (s)) f'(s) ds$ and 
$$\left|  K(|x|) g(f (t)) f'(t) 	\right| \leq C K(|x|) \min \left\{ \left| f(t)\right|
^{q_{1}-1},\left| f(t) \right| ^{q_{2}-1}\right\} \leq C  K(|x|)\min \left\{ \left| t\right|
^{q_{1}-1},\left| t\right| ^{q_{2}-1}\right\} .$$
Then we can apply the results in \cite{BPR} (in particular Proposition 3.8) and the fact that $E \hookrightarrow L_{K}^{q_{1}}(\mathbb{R}^{N})+L_{K}^{q_{2}}(\mathbb{R}^{N})$  (see Theorem \ref{THM(cpt)}), to get that also $I_3$ is $C^1$ in $E$, with differential given by 
$$I'_3 (u)h = \int_{\mathbb{R}^{N}}K(|x|) g(f (u)) f'(u) h\,dx.$$
As to $I_2$, we can  repeat the arguments of proposition (2.9) of \cite{Uberlandio1}, which work also in our hypotheses, to get that $I_2$ is well defined, continuous and Gateaux differentiable, with differential $I'_2$ given by
$$  I'_2 (u) h = \int_{\mathbb{R}^{N}}V(|x|) f (u) f'(u) h  dx .$$
In order to conclude, we need to prove that the map $I'_2 : E \rightarrow E' $ is continuous. Let $\{ u_n \}_n$ be a sequence in $E$ with $u_n \rightarrow u$ in $E$. Define
$$ \alpha_n = ||I'_2(u_n ) - I'_2(u)||_{E'}= \sup_{||h||\leq 1} \left| \left( I'_2 (u_n ) - I'_2 (u) \right) h \right| = 
\sup_{||h||\leq 1} \left|  \int_{\mathbb{R}^{N}} V(|x|) \left( f(u_n ) f' (u_n ) - f (u) f'(u) \right)  h  dx \right|.$$ 
We claim that $\alpha_n \rightarrow 0$. In proving this, we will use $C$ to indicate different positive constants, that can change from line to line but are independent from $h$ and $n$. We notice as first thing that $(1)$ of Lemma \ref{lem:properties} implies 
$$\sup_{||h||\leq 1}\left\{ \int_{\mathbb{R}^{N}} V(|x|) f(h)^2 dx \right\} \leq C .$$
Then we compute
$$\left|  \int_{\mathbb{R}^{N}} V(|x|) \left( f(u_n ) f' (u_n ) - f (u) f'(u) \right)  h  dx \right| \leq   \int_{B_1} V(|x|) \left| f(u_n ) f' (u_n ) - f (u) f'(u) \right|  |h  | dx $$
$$+   \int_{B_1^c} V(|x|) \left| f(u_n ) f' (u_n ) - f (u) f'(u) \right| | h | dx .$$
Recalling Corollary \ref{COR:est}, we get $|h(x)| \leq C$ in $B_1^c $ for all $h$ with $||h|| \leq 1$, and we can assume $C>1$. Hence from (7) of Lemma \ref{change2} we derive
$$ |h(x) | = C \, \frac{|h(x)|}{C}  \leq C f \left(\frac{|h(x)|}{C} \right)  \leq C f \left(|h(x)| \right) .$$
From this, applying H\"older inequality, we get
$$\int_{B_1^c} V(|x|) \left| f(u_n ) f' (u_n ) - f (u) f'(u) \right| | h | dx \leq 
C \int_{B_1^c} V(|x|) \left| f(u_n ) f' (u_n ) - f (u) f'(u) \right| f \left(|h(x)| \right) dx $$
$$\leq
C \left( \int_{B_1^c} V(|x|) \left| f(u_n ) f' (u_n ) - f (u) f'(u) \right|^2  dx \right)^{1/2} \, 
\left( \int_{B_1^c} V(|x|) f\left(|h(x)| \right)^2 dx \right)^{1/2} $$
$$\leq
C  \left( \int_{B_1^c} V(|x|) \left| f(u_n ) f' (u_n ) - f (u) f'(u) \right|^2  dx \right)^{1/2} $$
As $u_n \rightarrow u$ in $D_r^{1,2} \left(\mathbb{R}^{N}  \right) $, we can assume, up to a subsequence, that $u_n (x) \rightarrow u(x)$ for a.e. $x \in 
\mathbb{R}^{N} $. Also, from $(2)$ of Lemma \ref{lem:properties}, we deduce that $V^{\frac{1}{2}} f(u_n ) \rightarrow V^{\frac{1}{2}} f(u) $ in $L^2 (\mathbb{R}^{N})$ and hence, up to a subsequence, we can assume $V f(u_n )^2 \leq k \in L^1 (\mathbb{R}^{N}) $. Hence we have $V(|x|) \left| f(u_n ) f' (u_n ) - f (u) f'(u) \right| \rightarrow 0$ a.e. and
$$V(|x|) \left| f(u_n ) f' (u_n ) - f (u) f'(u) \right|^2 \leq C V(|x|) \left[ f(u_n )^2 f' (u_n )^2 + f(u)^2 f'(u)^2 \right ] $$
$$\leq C k + C V f(u)^2 \in  L^1 (\mathbb{R}^{N}),$$
\noindent so, by Dominated Convergence Theorem, we have 
$$\int_{B_1^c} V(|x|) \left| f(u_n ) f' (u_n ) - f (u) f'(u) \right|^2  dx  \rightarrow 0$$
\noindent which implies 
$$\sup_{||h||\leq 1} \int_{B_1^c} V(|x|) \left| f(u_n ) f' (u_n ) - f (u) f'(u) \right| | h | \, dx \rightarrow 0 .$$
\noindent On the other hand, by the hypothesis on $V$ and Lemma \ref{COR:est}, we get
$$
\int_{B_1} V(|x|) \left| f(u_n ) f' (u_n ) - f (u) f'(u) \right| \, |h|\, dx \leq C \int_{B_1} \frac{1}{|x|^2}\left| f(u_n ) f' (u_n ) - f (u) f'(u) \right|
\frac{1}{|x|^{  \frac{N-2}{2}  }}  \, dx =$$
$$ C \int_{B_1} \left| f(u_n ) f' (u_n ) - f (u) f'(u) \right|
\frac{1}{|x|^{\frac{N}{2}+1}} \, dx .
$$
\noindent As $N/2+1 <N$ we have $|x|^{-N/2-1} \in L^1 (B_1 )$, while $\left| f(u_n ) f' (u_n ) - f (u) f'(u) \right|\rightarrow 0$ a.e. in $\mathbb{R}^N$ and 
$\left| f(u_n ) f' (u_n ) - f (u) f'(u) \right| \leq C $ because of (9) of  Lemma \ref{change2}. Again by Dominated Convergence Theorem we get 
$$\int_{B_1} \left| f(u_n ) f' (u_n ) - f (u) f'(u) \right|
\frac{1}{|x|^{\frac{N}{2}+1}} dx \rightarrow 0$$
\noindent and hence 
$$\sup_{||h||\leq 1} \int_{B_1} V(|x|) \left| f(u_n ) f' (u_n ) - f (u) f'(u) \right| | h | \, dx \rightarrow 0 .$$
\noindent This holds for a subsequence of any sequence $u_n \rightarrow u$, and from this it is easy to get the thesis.
\end{proof}

According to the above result, a critical point $u$ of $I$ satisfies $I'(u)h=0$, that is
\begin{equation}
\label{EQweakdual}
\int_{\mathbb{R}^{N}}\nabla u \nabla h \, dx + \int_{\mathbb{R}^{N}}V(|x|) f (u) f'(u) h \,  dx -\int_{\mathbb{R}^{N}}K(|x|) g(f (u)) f'(u) h\,dx =0
\end{equation}
for all $h \in E$. This is, of course, a weak formulation of equation (\ref{EQdual}). We now want to show that a critical point $u$ of $I$ is a classical solution of equation (\ref{EQdual}) in $ \mathbb{R}^{N} \backslash \{ 0 \}$.

\begin{thm} \label{THM:u_classical}
	Assume the hypotheses of Theorem \ref{THM:critical}. Let $u$ be a critical point of $I$. Then $u \in  C^2 ( \mathbb{R}^{N} \backslash \{ 0 \} )$ and $u$ is a classical solution of equation (\ref{EQdual}) in $ \mathbb{R}^{N} \backslash \{ 0 \}$.

\end{thm} 	
\begin{proof} We deal with radial functions and for them, with a little abuse of notation, we will write $u(x) = u(|x|) = u(r)$ for $r= |x|$, so identifying $u$ with a function defined a.e. on $\mathbb{R}_+$. Using  this trick, the integral equation (\ref{EQweakdual}) becomes an integral equation in dimension 1, that is
\begin{equation}
\label{EQweakdualone}
\int_{0}^{+\infty}  u' (r) h' (r) \, r^{N-1} dr + \int_{0}^{+\infty}  V(r) f (u(r)) f'(u(r)) h(r) \,   r^{N-1} dr
\end{equation}
$$
-\int_{0}^{+\infty}  K(r) g(f (u(r) )) f'(u(r) ) 
h(r) \, r^{N-1} dr =0
$$
for all $h \in  E$. Of course equation (\ref{EQweakdualone}) can be considered as a weak formulation of the following ODE:
\begin{equation}
\label{EQdualone}
u'' + \frac{N-1}{r} u' + V(r) f(u) f'(u) - K(r) g (f(u)) f'(u) =0 \quad {\mbox {in}} \quad \mathbb{R}_+.
\end{equation}
We will now prove that $u$ is a classical solution of (\ref{EQdualone}). To be precise, we will prove the following claim.
\begin{quote}
{\it Claim:}
fix any $0<a<b< + \infty$ and let $I=(a,b)$. Then $u \in C^2 (I)$ and $u$ is a classical solution of (\ref{EQdualone}) in $I$. 
\end{quote}

\noindent The proof of the claim in divided in three steps: 
	\par \noindent (i) $u \in H^1 (I)$;
		\par \noindent (ii) $u\in H^2 (I)$;
		\par \noindent (iii) $u$ is a classical solution of (\ref{EQdualone}) in $I$.
	
\par \noindent Step (i) is easily obtained with the same argument of Lemma 27 in \cite{BGR_II}. 
We now prove (ii). Let us take any $\varphi \in C_{\mathrm{c}}^{\infty }(I)$. Take $\delta >0$ such that $a-\delta >0$ and define $I_{\delta} = (a-\delta , b+\delta )$. Of course $\varphi \in C_{\mathrm{c}}^{\infty }(I_{\delta})$. Let $\psi \in C^{\infty}(\mathbb{R})$ such that $0\leq \psi \leq 1$, $\psi (r) =0$ if $r\leq a -\delta/2$ and $\psi (r) =1$ if $r\geq a -\delta/3$. Define also
	$$
	v(r) := \int_{a}^{r} \frac{ \varphi '(s)}{s^{N-1}}\, ds = \frac{ \varphi (r)}{r^{N-1}}+ (N-1) \int_{a}^{r} \frac{ \varphi (s)}{s^{N}}\, ds.
	$$ 
Notice that $v \in C^{\infty}(I_{\delta})$ and $v(r)= 0 $ if $r \in (a-\delta , a)$. Let $\varepsilon >0$ be such that supp\,$\varphi \subset (a+ \varepsilon, b- \varepsilon)$. Then for $r \in (b-\varepsilon , b+ \delta)$ one has	
	$$
	v(r)= v(b)= (N-1)\int_a^{b} \frac{ \varphi (s)}{s^{N}}\, ds =: {\overline v} . 
	$$
Define now $w(r) = v(r) -{\overline v} \psi (r)$. Then $w \in  C^{\infty}(I_{\delta})$ and supp\,$w \subseteq [a -\frac{\delta}{2} , b- \varepsilon ]$, whence $w \in C_{\mathrm{c}}^{\infty }(I_{\delta})$. Hence by the equation  (\ref{EQweakdualone}) we easily get that
	$$
	\int_{I_{\delta}} u' (r) w' (r) r^{N-1} dr = \int_{I_{\delta}} \eta (r) w(r) dr, 
	$$
where $\eta \in L^{\infty} (I_{\delta})$. As $w' (r) = \varphi ' (r) r^{1-N}- {\overline v} \psi ' (r)$, from the above equation we deduce
	$$
	\int_{I} u' (r) \varphi ' (r)  dr = \int_{I_{\delta}} \eta (r) w(r) dr +	{\overline v} \int_{I_{\delta}} u' (r) \psi ' (r) \, r^{N-1} dr =
	\int_{I_{\delta}} \eta (r) \frac{ \varphi (r)}{r^{N-1}} dr $$
	$$+(N-1) \int_{I_{\delta}} \eta (r)  \left( \int_{a}^{r} \frac{ \varphi (s)}{s^{N}}\, ds \right) dr  + {\overline v} \int_{I_{\delta}} u' (r) \psi ' (r)  
	\, r^{N-1}  dr .
	$$
It is easy to see that the following estimates hold, with a constant $C>0$ depending on $u$ but not on $\varphi$:
	$$\left| 	\int_{I_{\delta}} \eta (r) \frac{ \varphi (r)}{r^{N-1}} dr \right| \leq C || \varphi ||_{L^2 (I)}, \quad 
	\left| (N-1) \int_{I_{\delta}} \eta (r)  \left( \int_{a}^{r} \frac{ \varphi (s)}{s^{N}}\, ds \right) dr   \right| \leq C || \varphi ||_{L^2 (I)}, $$
	$$
	\left|{\overline v} \int_{I_{\delta}} u' (r) \psi ' (r)  dr \right|  \leq C || \varphi ||_{L^2 (I)}.
	$$
The last inequality derives from the definition of ${\overline v}$ and the fact that $\left| \int_{I_{\delta}} u' (r) \psi ' (r)  dr \right|\leq C ||u||$, where $||u||$ is the norm of $u$ in $E$. Then we get
	$$
	\left| \int_{I} u' (r) \varphi ' (r)  dr \right| \leq C || \varphi ||_{L^2 (I)}.
	$$
As this holds for every $\varphi   \in C_{\mathrm{c}}^{\infty }(I)$, standard Sobolev space theory gives $u' \in H^1 (I)$ and hence $u \in H^2 (I)$.
As to (iii), once we have $u \in H^2 (I)$, it is a standard task to get that $u \in C^2 (I)$ and that the equation (\ref{EQdualone}) is satisfied in the classical sense in $I$.
This concludes the proof of the claim.
	
	Now it is easy to get the thesis of the theorem. The claim holds for every $I= (a,b)$ with $0<a<b<+\infty$, hence we have that $u \in  C^2 ( \mathbb{R}_+)$ and that it satisfies equation (\ref{EQdualone}) in $\mathbb{R}_+$. Coming back to dimension $N$, it is then obvious that $u \in C^2 ( \mathbb{R}^N \backslash \{ 0 \} )$ and 	
	\begin{equation}
	-\Delta u + V(|x|) f(u) f'(u) = K(|x|) g(f(u)) f'(u)
	\quad\text{in }\mathbb{R}^N \backslash \{ 0 \}.
	\end{equation}
	
\end{proof}

Now we show that $w= f(u)$ is a classical solution of equation (\ref{EQg}) in $\mathbb{R}^{N} \backslash \{ 0\}$.

\begin{thm}
Assume the hypotheses of Theorem \ref{THM:critical}. Let $u\in E$ be a critical point of $I$ and set $w= f(u)$. Then $w \in C^2 \left(\mathbb{R}^{N} \backslash \{ 0\} \right) $ and it is a classical solution of equation (\ref{EQg}) in $\mathbb{R}^{N} \backslash \{ 0\}$.

\end{thm}

\begin{proof}
From Theorem \ref{THM:u_classical} and the fact that $f\in C^{\infty}$, it is obvious that $ w \in C^2 \left(\mathbb{R}^{N} \backslash \{ 0\} \right) $. Direct computations then show that
	$$ \Delta w + w \Delta \left( w^2 \right) = \frac{1}{f'(u)}\Delta u .$$
It is then easy to get the result by substituting in equation (\ref{EQdualone}).
\end{proof}

	To complete our analysis, we now prove that if $u$ is a critical point of $I$ and $w=f(u)$, then $w$ also satisfies (\ref{EQweak}), that is, $w$ is a weak solution of (\ref{EQg}).

\begin{thm}

Assume the hypotheses of Theorem \ref{THM:critical}. Let $u\in E$ be a critical point of $I$ and set $w= f(u)$. Then $w =f(u) \in X$ and for all $h \in C_{\mathrm{c}, r}^{\infty }( \mathbb{R}^{N} ) $ one has
\begin{equation}
\int_{\mathbb{R}^{N}}  \left( 1+ 2w^2 \right) \nabla w\cdot \nabla h\, dx+ \int_{\mathbb{R}^{N}}   2w |\nabla w|^2 \, h \, dx + \int_{
\mathbb{R}^{N}}V\left( \left| x\right| \right) wh\,dx=\int_{
\mathbb{R}^{N}}K(|x|) g(w)  h\,dx 
\label{EQweak2}
\end{equation}

\end{thm}

\begin{proof}
It is obvious by our definitions that $\int_{\mathbb{R}^{N}} V \left( |x| \right) \, w^2 \, dx < + \infty .$ Moreover, we have $\nabla w = f'(u) \nabla u$ and thus $\int_{\mathbb{R}^{N}} |\nabla w|^2 dx \leq \int_{\mathbb{R}^{N}} |\nabla u|^2 dx< + \infty $. This gives $w \in X$. 
To prove (\ref{EQweak2}), we start by noticing that easy computations give
$$
\left( f^{-1} \right)'(t) = \sqrt{ 1 + 2 t^2 } , \quad \quad \left( f^{-1}\right)''(t) = \frac{2t}{ \sqrt{ 1 + 2 t^2 } }.
$$
Hence, as $u= f^{-1}(w)$, we derive $\nabla u = \left( f^{-1} \right)'(w)\, \nabla w  = \sqrt{ 1 + 2 w^2 } \, \nabla w $. Let us now fix $h \in C_{\mathrm{c}, r}^{\infty }( \mathbb{R}^{N} ) $ and define $\varphi = \left( f^{-1} \right)'(w)\, h  = \sqrt{ 1 + 2 w^2 }  \, h$. We want to prove that $\varphi \in E$. Of course $\varphi $ is radial, so what we actually need to prove are the following statements:
\begin{itemize}

\item[$i)$] $\int_{\mathbb{R}^{N}} V \left( |x| \right) \, f\left( \varphi \right)^2 dx <+ \infty $;

\item[$ii)$] $\int_{\mathbb{R}^{N}} |\nabla \varphi |^2 \, dx <+ \infty $.
\end{itemize}

In order to prove $i)$, we use the properties of $f$ and $ f^2$ (see lemma \ref{change2}). In particular, from $(11)$ it is easy to obtain that for all $C>1$ there is a constant $k = k(C) >0$ such that $f(C t)^2\leq k f(t)^2 $ for al $t>0$. Recalling that $h \in C_{\mathrm{c}, r}^{\infty }( \mathbb{R}^{N} ) $ we can assume 
$|h(x)| \leq C$ and $\mathrm{supp\,} h \subseteq B_R$. Hence we can compute
$$
\int_{\mathbb{R}^{N}} V \left( |x| \right) \, f\left( \varphi \right)^2 dx = \int_{B_R} V \left( |x| \right) \, f\left(| \varphi |\right)^2 dx =
 \int_{B_R} V \left( |x| \right) \, f\left(  \sqrt{ 1 + 2 w^2 }\, | h|\right)^2 dx 
$$
$$\leq  \int_{B_R} V \left( |x| \right) \, f\left(  \sqrt{ 1 + 2 w^2 } \, C\right)^2 dx \leq k(C) \int_{B_R} V \left( |x| \right) \, f\left(  \sqrt{ 1 + 2 w^2 } \right)^2 dx 
$$
$$
=k(C)\,\int_{B_{R }\cap \left\{  |w| \geq 1 \right\}} V \left( |x| \right) \, f\left(  \sqrt{ 1 + 2 w^2 } \right)^2 dx +
k(C)\, \int_{B_{R }\cap \left\{  |w| < 1 \right\} } V \left( |x| \right) \, f\left(  \sqrt{ 1 + 2 w^2 } \right)^2 dx.
$$
On the one hand, we easily get
$$
\int_{B_{R }\cap \left\{  |w| < 1 \right\} } V \left( |x| \right) \, f\left(  \sqrt{ 1 + 2 w^2 } \right)^2 dx \leq 
\int_{B_{R }\cap \left\{  |w| < 1 \right\} } V \left( |x| \right) \, f\left(  \sqrt{3 } \right)^2 dx
$$
$$
\leq f\left(  \sqrt{3 } \right)^2 \int_{B_{R } } V \left( |x| \right)  \, dx <+\infty .
$$
On the other hand, when $|w| \geq 1 $ one has $ \sqrt{ 1 + 2 w^2 } \leq 2 |w |$ and hence
$$
\int_{B_{R }\cap \left\{  |w| \geq 1 \right\}} V \left( |x| \right) \, f\left(  \sqrt{ 1 + 2 w^2 } \right)^2 dx \leq
\int_{B_{R }\cap \left\{  |w| \geq 1 \right\}} V \left( |x| \right) \, f\left(  2|w| \right)^2 dx
$$
$$
\leq k \int_{B_{R }\cap \left\{  |w| \geq 1 \right\}} V \left( |x| \right) \, f\left(  |w| \right)^2 dx
\leq k \int_{B_{R } } V \left( |x| \right) \,  |w|^2 \, dx <+ \infty 
$$
because $w \in X$. So $i)$ is proved. 

As to $ii)$, we compute
$$
\nabla \varphi = \sqrt{ 1 + 2 w^2 } \, \, \nabla h + 2 \, h \, \frac{w}{\sqrt{ 1 + 2 w^2 } } \, \nabla w 
$$
and we easily get
$$ \left| 2 \, h \, \frac{w}{\sqrt{ 1 + 2 w^2 } } \, \nabla w   \right| \leq C |\nabla w | \in L^2 \left( \mathbb{R}^{N}  \right).$$
On the other hand, as $w \in D_r^{1,2} \left(\mathbb{R}^{N}  \right) $, we have $w \in L_{loc}^2\left(\mathbb{R}^{N}  \right) $ and hence
$$
\int_{\mathbb{R}^{N}} \left| \sqrt{ 1 + 2 w^2 } \, \, \nabla h \right|^2 \, dx \leq C \int_{B_R} \left(  1 + 2 w^2 \right) \, dx  < + \infty .
$$
So also $ii)$ is proved.

We now conclude the proof of the lemma. As $\varphi \in E$ and $I' (u)=0$, exploiting the computations above we get 
$$
0= I'(u) \varphi = \int_{\mathbb{R}^{N}}\sqrt{ 1 + 2 w^2 } \, \, \nabla w \cdot \sqrt{ 1 + 2 w^2 } \, \nabla h \, dx +
\int_{\mathbb{R}^{N}}\sqrt{ 1 + 2 w^2 } \, \nabla w \cdot  \frac{2hw}{\sqrt{ 1 + 2 w^2 } } \, \nabla w \, dx $$
$$
+\int_{\mathbb{R}^{N}} V f(u) f' (u) \left( f^{-1} \right)' (w) h \, dx -
\int_{\mathbb{R}^{N}}Kg(f(u)) f' (u) \left( f^{-1} \right)' (w) h \, dx$$
$$
=\int_{\mathbb{R}^{N}} \left( 1 + 2 w^2 \right) \, \nabla w \cdot  \nabla h \, dx +
\int_{\mathbb{R}^{N}}  2hw \, \left| \nabla w \right|^2 \, dx+
\int_{\mathbb{R}^{N}} V \, w \, h \, dx -
\int_{\mathbb{R}^{N}} K\, g(w) \, h \, dx.
$$
Therefore (\ref{EQweak2}) is satisfied and the theorem is proved.
\end{proof}


\section{Existence of solutions \label{SEC: ex}}

This section is devoted to our main existence result, which is the following.

\begin{thm}
\label{THM:ex} Assume $N\geq 3$, $\left( \mathbf{H}\right) $ and that $g: {\mathbb{R}} \rightarrow {\mathbb{R}}$ is a continuous function satisfying ${\bf \left( g_{1}\right) } $, ${\bf \left( g_{2}\right) } $, $  \left({\bf g}_{q_{1},q_{2}}\right) $. Assume the hypotheses of 
Theorems \ref{THM0} and \ref{THM1} with $q_1 ,q_2 $ satisfying respectively (\ref{th1}) and (\ref{th2}).
\noindent Then the functional $I:E\rightarrow \mathbb{R}$ has a nonnegative
critical point $u\neq 0$.
\end{thm}

\begin{rem}
\label{RMK:thm:ex} In Theorem \ref{THM:ex}, as we look for non negative solutions, we can assume $g(t)=0$ for all $t\leq 0$. Indeed, if we have a nonlinearity $g$ satisfying the hypotheses, we can replace $g$ with $\chi _{\mathbb{R}_{+}}\left( t\right) g\left(
t\right) $ ($\chi _{\mathbb{R}_{+}}$ is the characteristic function of $\mathbb{R
}_{+}$), and the new nonlinearity still satisfies the hypotheses.

\end{rem}

\begin{rem}

Thanks to Theorems \ref{THM0} and \ref{THM1}, the hypotheses of Theorem \ref{THM:ex} imply that $E$ is compactly embedded into $L_{K}^{q_{1}}(\mathbb{R}^{N})+L_{K}^{q_{2}}(\mathbb{R}^{N})$. This is one of the main devices to get our existence result.

\end{rem}

\begin{rem}
As concerns examples of nonlinearities satisfying the hypotheses of Theorem \ref{THM:ex}, the simplest 
$g\in C\left(\mathbb{R};\mathbb{R}\right) $ such that $ \left({\bf  g}_{q_{1},q_{2}}\right) $ holds is 
\[g\left( t\right) =\min \left\{ \left| t\right| ^{q_{1}-2}t,\left| t\right|
^{q_{2}-2}t\right\} ,
\]
which also ensures ${\bf  \left( g_{1}\right) }$ if $
q_{1},q_{2}>4$ (with $\theta = \min \left\{ \frac{q_1}{2}, \frac{q_2}{2} \right\} $). Another model example is 
\[
g\left( t\right) =\frac{\left| t\right| ^{q_{2}-2}t}{1+\left| t\right|
^{q_{2}-q_{1}}}\quad \text{with }1<q_{1}\leq q_{2},
\]
which ensures ${\bf \left( g_{1}\right) }$ if $q_{1}>4$ (with $\theta =\frac{q_1}{2}$). Note that, in
both these cases, also ${\bf \left( g_{2}\right) }$ holds true. Moreover, both of these functions $g$
become $g\left( t\right) =\left| t\right| ^{q-2}t$ if $q_{1}=q_{2}=q$. 
\end{rem}

We will get Theorem \ref{THM:ex} by applying a version of the well-known Mountain-Pass Lemma (see chapter 2 in \cite{AubinEkeland}). Let us first recall  the so-called Palais-Smale condition.

\begin{defin}[Palais-Smale condition]
Let $Y$ be a Banach space and $\Phi : Y \rightarrow \mathbb{R}$ a $C^1$ functional. We say that $\Phi $ satisfies the Palais-Smale condition if for any sequence 
$\{ x_n \}_n $ sucht that $\Phi(x_n )$ is bounded in $\mathbb{R}$ and $\Phi' (x_n ) \rightarrow 0$ in $Y'$, there is a subsequence $\{ x_{n_k} \}_k $ converging in $Y$.
\end{defin}

\begin{thm}[Mountain Pass Lemma]
\label{MPLemma}
\par \noindent Let $Y$ be a Banach space and $\Phi : Y \rightarrow \mathbb{R}$ a $C^1$ functional with $\Phi (0)=0$ . Assume that $\Phi$ satisfies the Palais-Smale condition and that there are a subset $S \subseteq Y$ and a real number $\alpha >0$  such that:

\begin{itemize}

\item[(1)] $Y  \backslash S$ is not arcwise connected;

\item[(2)] $\Phi  (x) \geq \alpha$ for all $x \in S$;

\item[(3)] there exists $ y \in Y  \backslash (C_0 \cup S )$ such that $\Phi (y)<0$, where $C_0$ is the connected component of $ Y  \backslash S$ such that $0 \in C_0$. 

\end{itemize}

\noindent Then $\Phi $ has a critical point $u \in Y$ such that $\Phi (u) \geq \alpha $.

\end{thm}

To prove Theorem \ref{THM:ex} we will prove that the functional $I: E\rightarrow \mathbb{R}$ satisfies the hypotheses of the Mountain Pass Lemma. It is obvious that $I(0)=0$. The other hypotheses of the Mountain Pass Lemma are proved in the following lemmas. More precisely, assumptions {\it (1)} and {\it (2)} are proved in Lemma \ref{LE:MP1}, while assumption {\it (3)} is proved in Lemma \ref{LE:MP2}. In Lemmas \ref{LE:MP3} and \ref{LE:MP4} we show that $I$ satisfies the Palais-Smale condition. 

Recall the three functionals $I_1 , I_2 , I_3$ introduced in the proof of Lemma \ref{THM:critical} and define, for $u \in E$, 
$$J (u) = I_1 (u)+ I_2 (u)= \frac{1}{2} \int_{\mathbb{R}^{N}}|\nabla u |^2 \, dx  + \frac{1}{2} \int_{\mathbb{R}^{N}}V(|x|) f(u)^2 \,dx. $$
Then, for any $\rho >0$, define 
$$
S_{\rho}= \left\{ u \in E \, | \, J(u) = \rho \right\}.
$$
\begin{lem}
\label{LE:MP1}
\label{LEM:MPgeom}Assume the hypotheses of Theorem \ref{THM:ex}. Then there is $ \rho^* >0$ such that for all $\rho \in (0, \rho^* )$ the set $E\backslash S_{\rho}$ is not arcwise connected and there exists $\alpha = \alpha (\rho )>0$ such that $I(u) \geq \alpha$ for all $u \in S_{\rho}$.

\end{lem}

\proof
Fix any $v \in E \backslash \{ 0 \}$ and set $\rho_1 = J(v) >0$. Then for all $\rho \in (0, \rho_1 )$ the set $E\backslash S_{\rho}$ is not arcwise connected, because $J$ is a continuous functional on $E$ and any continuous path joining $0$ and $v$ must intersect $S_{\rho}$. To get $S=S_\rho$ and $\alpha $ as in Mountain Pass Lemma, we recall first that, by Corollary \ref{COR:embed}, one has $X \hookrightarrow E$ and therefore there exists $C>0$ such that $||u|| \leq C ||u||_X$ for all $u\in X$. Also, we know that if $u \in E$ then $f(u) \in X$, so that, for all $u \in E$, we have
$$
||f(u)|| \leq C ||f(u)||_X.
$$
In the hypotheses of Theorem \ref{THM:ex}, we can choose $0<R_1 <R_2$ such that $\mathcal{S}_{0}(q_1 ,R_1 ) < +\infty$ and $\mathcal{S}_{\infty}(q_2 ,R_2 ) < +\infty$. Hence
$$ \left| \int_{\mathbb{R}^N}K(|x|) G(f(u)) dx \right| \leq  M \int_{\mathbb{R}^N}K(|x|) \min \left\{ |f(u)|^{q_1}, |f(u)|^{q_2}   \right\} dx 
$$
$$\leq
M \int_{B_{R_1}} K(|x|)|f(u)|^{q_1}dx +M \int_{B_{R_2} \backslash B_{R_1}} K(|x|)|f(u)|^{q_1}dx +M \int_{B_{R_2^c}} K(|x|)|f(u)|^{q_2}dx
$$
$$ \leq
M \mathcal{S}_{0}(q_1 ,R_1 ) ||f(u)||^{q_1} + MC_{R_1 , R_2}||f(u)||^{q_1}+M \mathcal{S}_{\infty}(q_2 ,R_2 ) ||f(u)||^{q_2} .$$
These inequalities derive from the hypotheses on $g$, the definitions of  $\mathcal{S}_{0}$ and $\mathcal{S}_{\infty}$, and Lemmas \ref{Lem(corone)} and 
\ref{A6}. So we get 
$$ \left| \int_{\mathbb{R}^N}K(|x|) G(f(u)) dx \right| \leq C_1 ||f(u)||^{q_1}  + C_2 ||f(u)||^{q_2} \leq C_3||f(u)||_{X}^{q_1} + C_4 ||f(u)||_{X}^{q_2}.$$
Now we have 
$$||f(u)||_{X}^2= \int_{\mathbb{R}^N}|\nabla f(u) |^2 dx + \int_{\mathbb{R}^N}V(|x|) f(u)^2 dx \leq \int_{\mathbb{R}^N}|\nabla u |^2 dx + \int_{\mathbb{R}^N}V(|x|) f(u)^2 dx = 2 J(u).$$
and therefore, for $u \in S_{\rho}$, we get 
$$ \left| \int_{\mathbb{R}^N}K(|x|) G(f(u)) dx \right|\leq C_5 \rho^{q_1 /2} + C_6 \rho^{q_2 /2}.$$
Hence, for $u \in S_{\rho}$ we conclude that
$$I(u) = J(u) - \int_{\mathbb{R}^N}K(|x|) G(f(u)) dx \geq \rho -C_5 \rho^{q_1 /2} - C_6 \rho^{q_2 /2}.$$
As $\frac{q_1}{2} , \frac{q_2}{2} >2$, it is obvious that for $\rho >0$ small enough we have $\alpha = \alpha (\rho) = \rho -C_5 \rho^{q_1 /2} - C_6 \rho^{q_2 /2}>0$, and this concludes the proof of the lemma.
\endproof

\begin{lem}
\label{LE:MP2}
Take $\rho >0$ as in Lemma \ref{LE:MP1}. Then there exists $v\in E$ such that $J(v) > \rho$ and $I(v)<0$.
\end{lem}
\proof
From assumption ${ \bf\left( g_{1}\right)} $ and ${ \bf
\left( g_{2}\right) }$ we infer that $G(t)\geq 0$ for all $t$ and, for every $t_+ >t_0$ and all $t> t_+$, 
\begin{equation}
\label{diseG}
G\left( t\right) \geq \frac{G\left( t_+ \right) }{t_{+}^{2\theta }}
t^{2\theta }>0 .
\end{equation}
Clearly it is not restrictive to assume $t_0 \geq 1$. Now we fix $t_1 >1$ such that $f(t_1 )>t_0$ and then we pick a non negative
function $u_{0}\in C_{c,r}^{\infty }( \mathbb{R}^{N}  )$ such that the set $\{x\in \mathbb{R}^{N}:u_{0}\left(
x\right) \geq t_{1}\}$ has positive Lebesgue measure. Hence for every $\lambda >1$, using (\ref{diseG}) with $t_+ = f(t_1)$, we get 
$$
\int_{\mathbb{R}^{N}}K(|x|) G\left( f\left( \lambda u_{0}\right) \right) dx \geq
\int_{\left\{ \lambda u_{0}\geq t_{1}\right\} }K(|x|) G\left( f\left( \lambda u_{0}\right) \right) dx\geq \frac{G(f(t_1 ))}{f(t_ 1)^{2\theta }}
\int_{\left\{ \lambda u_{0}\geq t_{1}\right\} }K(|x|)  \left( f\left( \lambda u_{0} \right) \right)^{2\theta }dx $$
$$\geq C_1 \lambda ^{\theta } \int_{\left\{ u_{0}\geq
t_{1}\right\} }K(|x|) \, u_{0}^{\theta } dx= C_2 \lambda ^{\theta } ,
$$
where $ C_2 = C_1 \int_{\left\{ u_{0}\geq t_{1}\right\} }K(|x|) \,  u_{0}^{\theta }dx>0$. On the other hand
$$ \int_{\mathbb{R}^{N} } V(|x|)  f( \lambda u_{0})^2dx \leq \lambda^2 \int_{\mathbb{R}^{N} }V(|x|) \, u_{0}^2 dx ,$$
so that 
$$
I(\lambda u_0 ) = \frac{1}{2} \int_{\mathbb{R}^{N} } \left| \lambda \nabla u_0 \right|^2 dx + \frac{1}{2} \int_{\mathbb{R}^{N} } V(|x|) f\left( \lambda u_0 \right)^2dx-
\int_{\mathbb{R}^{N}}K(|x|) G\left( f\left( \lambda u_{0}\right) \right) dx \leq C_3 \lambda^2 - C_2 \lambda ^{\theta }.
$$

\noindent As $\theta  >2$, we deduce that $I(\lambda u_0 ) \rightarrow - \infty $ when $\lambda \rightarrow +\infty$. As it is obvious that $J(\lambda u_0 ) \rightarrow + \infty $ when $\lambda \rightarrow +\infty$, the proof is concluded by choosing $v = \lambda u_0$ for $\lambda$ large enough.
\endproof

\begin{lem}
\label{LE:MP3}
Under the assumptions of Theorem \ref{THM:ex}, let $\left\{ u_n   \right\}_n \subseteq E$ be a Palais-Smale sequence for $I$, that is, a sequence such that 
$ \left\{ I(u_n ) \right\}_n$ is bounded and $ I' (u_n ) \rightarrow 0 $ in $E'$. Then $\left\{ u_n   \right\}_n $ in bounded in $ E$.
\end{lem}

\proof
We start with the following computation:
$$
I(u_n ) - \frac{1}{\theta} I'(u_n ) u_n = \left( \frac{1}{2}-\frac{1}{\theta}\right)  \int_{\mathbb{R}^{N}}| \nabla u_n |^2 dx + \frac{1}{2} \int_{\mathbb{R}^{N}}V(|x|) f(u_n )^2 dx -
\int_{\mathbb{R}^{N}}K(|x|) G(f (u_n )) dx $$
$$
-\frac{1}{\theta} \int_{\mathbb{R}^{N}}V(|x|) f (u_n )f'(u_n ) u_n dx + \frac{1}{\theta} \int_{\mathbb{R}^{N}}K(|x|) g(f (u_n ))f' (u_n ) u_n dx.
$$
Since {\it (5)} of Lemma \ref{change2} implies $f(t)^2 - f(t) f'(t) t \geq 0$ for all $t$, we have 
$$\int_{\mathbb{R}^{N}}V(|x|)\left(  f(u_n )^2- f (u_n )f'(u_n ) u_n \right) dx \geq 0$$
and this implies
$$
\frac{1}{2} \int_{\mathbb{R}^{N}}V(|x|) f(u_n )^2 dx -\frac{1}{\theta} \int_{\mathbb{R}^{N}}V(|x|) f (u_n )f'(u_n ) u_n dx =
 \left( \frac{1}{2}-\frac{1}{\theta}\right)  \int_{\mathbb{R}^{N}}V(|x|) f(u_n )^2 dx $$
$$+ \frac{1}{\theta}\int_{\mathbb{R}^{N}}V(|x|)\left(  f (u_n )^2- f (u_n )f'(u_n ) u_n \right) dx \geq  \left( \frac{1}{2}-\frac{1}{\theta}\right)  \int_{\mathbb{R}^{N}}V(|x|) f(u_n )^2 dx. 
$$
On the other hand, using the hypotheses on $g$ and {\it (5)} of Lemma \ref{change2} again, we have 
$$ \frac{1}{\theta} \int_{\mathbb{R}^{N}}K(|x|) g(f (u_n ))f' (u_n ) u_n dx -\int_{\mathbb{R}^{N}}K(|x|) G(f (u_n )) dx $$
$$\geq \
\frac{1}{\theta} \int_{\mathbb{R}^{N}}K(|x|) g(f (u_n ))f' (u_n ) u_n dx-  
\frac{1}{2\theta} \int_{\mathbb{R}^{N}}K(|x|) g(f (u_n ))f (u_n ) dx$$
$$\geq
\frac{1}{2\theta} \int_{\mathbb{R}^{N}}K(|x|) g(f (u_n ))f (u_n ) dx- \frac{1}{2\theta} \int_{\mathbb{R}^{N}}K(|x|) g(f (u_n ))f (u_n ) dx=0.
$$
Therefore we get
$$
I(u_n ) - \frac{1}{\theta} I'(u_n ) u_n \geq \left( \frac{1}{2}-\frac{1}{\theta}\right) \left( \int_{\mathbb{R}^{N}}| \nabla u_n |^2 dx +  \int_{\mathbb{R}^{N}}V(|x|) f(u_n ^2) dx \right).
$$
By definition, we have $\int_{\mathbb{R}^{N}}| \nabla u_n |^2 dx = ||u_n ||_{1,2}^2$ and $\int_{\mathbb{R}^{N}}V(|x|) f(u_n)^2 dx +1\geq ||u_n ||_o $. Hence, if $||u_n ||_{1,2} \leq 1$ we get
$$\int_{\mathbb{R}^{N}}| \nabla u_n |^2 dx + \int_{\mathbb{R}^{N}}V(|x|) f(u_n )^2 dx \geq  ||u_n ||_o -1 \geq ||u_n ||_o  +||u_n ||_{1,2} -2 = ||u_n || -2.$$
On the other hand, if  $||u_n ||_{1,2} > 1$ then $||u_n ||_{1,2}^2 > ||u_n ||_{1,2}$ and hence
$$\int_{\mathbb{R}^{N}}| \nabla u_n |^2 dx + \int_{\mathbb{R}^{N}}V(|x|) f (u_n )^2 dx \geq ||u_n ||_{1,2}+||u_n ||_{o} -1\geq ||u_n || -2.$$
So in any case we conclude 
$$
I(u_n ) - \frac{1}{\theta} I'(u_n ) u_n \geq \left( \frac{1}{2}-\frac{1}{\theta}\right)||u_n || -2\left( \frac{1}{2}-\frac{1}{\theta}\right).
$$
As $\left\{ u_n   \right\}_n $ is a Palais-Smale sequence, we can assume $I (u_n ) \leq C$ and we can fix $\delta >0$ such that 
$\delta < \theta  \left( \frac{1}{2}-\frac{1}{\theta}\right)$ and $\left| I' (u_n )u_n \right| \leq \delta ||u_n ||$ for large $n$'s. Hence we get 
$$C + \frac{\delta }{\theta}  ||u_n || \geq I(u_n ) - \frac{1}{\theta} I'(u_n ) u_n \geq \left( \frac{1}{2}-\frac{1}{\theta}\right)||u_n || -2\left( \frac{1}{2}-\frac{1}{\theta}\right),$$
that is,
$$C +2\left( \frac{1}{2}-\frac{1}{\theta}\right) \geq \left( \frac{1}{2}-\frac{1}{\theta}  - \frac{\delta }{\theta} \right)||u_n ||.$$
As $\frac{1}{2}-\frac{1}{\theta}  - \frac{\delta }{\theta} >0$, this implies that $\left\{ ||u_n ||  \right\}_n $ in bounded.
\endproof

\begin{lem}
\label{LE:MP4}
\label{LEM:PS}Under the assumptions of Theorem \ref{THM:ex}, the functional $
I:E\rightarrow \mathbb{R}$ satisfies the Palais-Smale condition.
\end{lem}

\proof

Let $\left\{ u_{n}\right\}_n $ be a sequence in $E$ such that $\left\{
I\left( u_{n}\right) \right\}_n $ is bounded and $I^{\prime }\left(
u_{n}\right) \rightarrow 0$ in $E^{\prime }$. By Lemma \ref{LE:MP3}, $\left\{ u_{n}\right\}_n $ is bounded in $E$ and therefore there exists a subsequence, that we still call $\left\{ u_{n}\right\}_n $, such that $u_n \rightharpoonup u$ in $D_r^{1,2} \left(\mathbb{R}^{N}\right) $ and $u_n (x) \rightarrow u(x)$ for a.e. $x$. Recall that we have introduced the three functionals $I_1 , I_2 , I_3$ (see Theorem \ref{THM:critical}) and we have defined $J= I_1 + I_2$, so that 
$I= J-I_3$. We know that $I_3$ is od class $C^1$ on $L_{K}^{q_{1}}+L_{K}^{q_{2}}$. By compactness of the embedding of $E$ into $L_{K}^{q_{1}}+L_{K}^{q_{2}}$, up to a subsequence we have that $u_n \rightarrow u$ in $L_{K}^{q_{1}}+L_{K}^{q_{2}}$, whence $I_{3}' (u_n ) \rightarrow I_{3}'(u)$ in the dual space of $L_{K}^{q_{1}}+L_{K}^{q_{2}}$ and $I_{3}' (u_n )(u-u_n )\rightarrow 0$ in $\mathbb{R}$.
We notice now that, as $f^2$ is a convex function, $J$ is a convex functional on $E$, so that
$$
J(u)-J(u_n ) \geq J'( u_n ) (u- u_n ) = I'( u_n ) (u- u_n ) + I_{3}'( u_n ) (u- u_n ).
$$
As $I'( u_n ) \rightarrow 0$ in $E'$ by hypothesis and $\left\{ u-u_{n}\right\} $ is bounded in $E$, we have $I'( u_n ) (u- u_n )\rightarrow 0$ and thus
$$
J(u) \geq J(u_n ) + o(1).
$$
Taking the liminf, this gives
\begin{equation}
\label{EQPS1}
\int_{\mathbb{R}^{N}}| \nabla u |^2 dx + \int_{\mathbb{R}^{N}}V(|x|) f(u)^2 dx \geq \liminf_n \left( \int_{\mathbb{R}^{N}}| \nabla u_n |^2 dx + \int_{\mathbb{R}^{N}}V(|x|) f(u_n )^2 dx \right) 
\end{equation}
$$\geq \liminf_n  \int_{\mathbb{R}^{N}}| \nabla u_n |^2 dx + \liminf_n \int_{\mathbb{R}^{N}}V(|x|) f(u_n )^2 dx .$$
By semicontinuity of the norm, we have
$$
\liminf_n \int_{\mathbb{R}^{N}}| \nabla u_n |^2 dx \geq \int_{\mathbb{R}^{N}}| \nabla u |^2 dx,
$$
so that (\ref{EQPS1}) gives
$$
\int_{\mathbb{R}^{N}}V(|x|) f(u )^2 dx \geq \liminf_n  \int_{\mathbb{R}^{N}}V(|x|) f(u_n )^2 dx.
$$
As Fatou's Lemma obviously implies $\int_{\mathbb{R}^{N}}V(|x|) f (u )^2 dx \leq \liminf_n  \int_{\mathbb{R}^{N}}V(|x|) f (u_n )^2 dx$, we deduce
\begin{equation}
\label{EQPS2}
\int_{\mathbb{R}^{N}}V(|x|) f (u )^2 dx = \liminf_n  \int_{\mathbb{R}^{N}}V(|x|) f (u_n )^2 dx.
\end{equation}
So, passing to a subsequence that we still label $\left\{ u_{n}\right\}_n $, we can assume
\begin{equation}
\label{EQPS3}
\int_{\mathbb{R}^{N}}V(|x|) f (u ) dx = \lim_n \int_{\mathbb{R}^{N}}V(|x|) f (u_n )^2 dx.
\end{equation}
\noindent Then by {\it (3)} of Lemma \ref{lem:properties} we get $||u- u_n ||_o \rightarrow 0$. Now, repeating the previous argument for this subsequence, we get again (\ref{EQPS1}), which now gives
$$
\int_{\mathbb{R}^{N}}| \nabla u |^2 dx \geq \liminf_n \int_{\mathbb{R}^{N}}| \nabla u_n |^2 dx 
$$
and hence 
$$
\int_{\mathbb{R}^{N}}| \nabla u |^2 dx = \liminf_n \int_{\mathbb{R}^{N}}| \nabla u_n |^2 dx . 
$$
Up to a subsequence again, we can assume
$$
\int_{\mathbb{R}^{N}}| \nabla u |^2 dx = \lim_n \int_{\mathbb{R}^{N}}| \nabla u_n |^2 dx . 
$$
Since $u_n \rightharpoonup u$ in $D_r^{1,2} \left(\mathbb{R}^{N}\right) $, we obtain that $u_n \rightarrow u$ in $D_r^{1,2} \left(\mathbb{R}^{N}\right) $, i.e., $||u- u_n ||_{1,2} \rightarrow 0$.
Hence $||u - u_n || = ||u- u_n ||_o + ||u- u_n ||_{1,2}\rightarrow 0$. 
\endproof


\proof[Proof of Theorem \ref{THM:ex}.] 
Taking $\rho$ and $v$ as in Lemmas \ref{LE:MP1} and \ref{LE:MP2}, we have $0=J(0) <\rho<J(v)$, so that  $v$ and $0$ are in two distinct 
connected components of $E\backslash S_{\rho}$. Hence, the previous lemmas show that all the hypoteses of Mountain Pass Lemma \ref{MPLemma} are satisfied, and thus we get a critical point $u$ of $I$, with $u \not= 0$. Let $u^{-}$ be the negative part of $u$. It is easy to see that $u^{-} \in E$, so that $I' (u) u^{-}=0$. The additional assumption $g(t)=0$ for $t<0$ implies 
\begin{equation}
\label{EQparteneg}
0= I^{\prime }\left( u\right) u^{-}=-\int_{\mathbb{R}^{N}}| \nabla u^{-} |^2 dx + \int_{\mathbb{R}^{N}}  V(|x|) f (u) f'(u) u^{-} dx,
\end{equation}
where, by the properties of $V$ and $f$, we have
$$
\int_{\mathbb{R}^{N}}  V(|x|) f (u) f'(u) u^{-} dx= \int_{\mathbb{R}^{N}}  V(|x|) f (-u^{-}) f'(u) u^{-} dx= -\int_{\mathbb{R}^{N}}  V(|x|) f (u^{-}) f'(u) u^{-} dx\leq 0.
$$
Hence (\ref{EQparteneg}) implies $\int_{\mathbb{R}^{N}}| \nabla u^{-} |^2 dx =0$. One concludes that $u^{-}=0$, because $u^{-}\in D_{r}^{1,2}\left(\mathbb{R}^{N} \right)$, and therefore $u$ is nonnegative.
\endproof

\section{Examples}\label{SEC:EX}

\noindent In this section we give some examples of application of our results, obtaining some existence results which are not included, as far as we know, in the previous literature. More precisely, we will make a comparison between our results and those of \cite{Uberlandio2}, which inspired the present study. In that paper the authors prove some existence results for equation (\ref{EQ}), assuming that $g$ grows like a power and that $V, K$ are controlled by suitable powers of $|x|$. Here we show some situations where the results of \cite{Uberlandio2} do not apply, while ours give existence of solutions. 

In all the examples, we will consider the model nonlinearity $g(t)= \min \{ t^{q_1 -1} , t^{q_2 -1} \}$ for simplicity, and we will let $4<q_1 \leq q_2$. 
As the throughout the paper, we will also assume $N\geq 3$ and hypothesis ${\bf (H)}$.

\begin{exa} \label{ex1}
Assume that there exist $c_1 ,c_2 , c_3 , c_4 >0$ such that
$$c_1 \, r^{2N}\leq K(r) \leq  c_2 \, r^{2N} \quad {\mbox as} \; r \rightarrow 0^+ , \quad \quad c_3 r^{3N} \leq  K(r) \leq \, c_4 r^{3N} \quad {\mbox as} \; r \rightarrow +\infty .$$
Computing the coefficients $b, b_0$ in \cite{Uberlandio2}, one gets $b_0=2N$ and $b\geq 3N$, so the results in \cite{Uberlandio2} cannot be applied, because they need $b_0\geq b$. If we let $\beta_0 = \beta_{\infty}=0$, $\alpha_0 = 2N$ and $\alpha_{\infty}=3N$ in Theorems \ref{THM0} and \ref{THM1}, we get 
$$ 
q_{0}^* (\alpha_0 , 0 )=2 \,\frac{\alpha_0 +N}{N-2}=  \frac{6N}{N-2} >6, \quad q_{\infty}^* (\alpha_{\infty},0)=2\, \frac{\alpha_{\infty} +N}{N-2}=  \frac{8N}{N-2}.
$$
Hence we can apply our existence results to nonlinearities $g(t)= \min \{ t^{q_1 -1} , t^{q_2 -1} \}$ with
$$4 < q_1 < \frac{6N}{N-2} <\frac{8N}{N-2}<q_2.$$
Notice that we are not assuming that $V$ has a power-like behavior at zero or at infinity (in this regard we just need hypothesis ${\bf (H)}$).

\end{exa}

\begin{exa} \label{ex2}
Let $N=3$ and assume that there exist $c_1 ,c_2 , c_3 , c_4 >0$ such that
$$\frac{c_1}{r^{1/2}}\leq K(r) \leq  \frac{c_2}{r^{1/2}} \quad {\mbox as} \; r \rightarrow 0^+ , \quad \quad  \frac{c_3}{r^{1/3}} \leq K(r) \leq \frac{c_4}{r^{1/3}} \quad {\mbox as} \; r \rightarrow +\infty .
$$
In this case the coefficients $b, b_0$ in \cite{Uberlandio2} are $b_0=- \frac{1}{2} $ and $b\geq -\frac{1}{3} $, so again $b>b_0$ and the results of 
\cite{Uberlandio2} cannot be applied. If we let $\beta_0 = \beta_{\infty}=0$, $\alpha_0 = - \frac{1}{2}$ and $\alpha_{\infty}=-\frac{1}{3}$ in Theorems \ref{THM0} and \ref{THM1}, we get 
$$ 
q_{0}^* (\alpha_0 , 0 )=2 \,\frac{\alpha_0 +3}{3-2}=  -1+6 =5, \quad q_{\infty}^* (\alpha_{\infty},0)=2\, \frac{\alpha_{\infty} +3}{3-2}=  \frac{16}{3}.
$$
Hence we can apply our existence results with
$$4 < q_1 < 5 <\frac{16}{3}<q_2 .$$
As for the previous example, the only assumption we need on the asymptotic behavior of $V$ is ${\bf (H)}$.

\end{exa} 

\begin{exa} \label{ex3}

Assume that there exist $c, \delta >0$ such that $V(r) \leq c \, e^{-\delta r}$ as $r \rightarrow +\infty$. The results of \cite{Uberlandio2} cannot be applied because they require $\liminf_{r \rightarrow \infty} \frac{V(r)}{r^a}>0$ for some $a \in \mathbb{R}$. Instead, we can give several existence results with different hypotheses on $K$. For example, assume $K(r)= r^N$. Then we take $\beta_0 =\beta_{\infty}=0$ and $\alpha_0 = \alpha_{\infty}=N$ in Theorems \ref{THM0} and \ref{THM1}, and we get 
$$
q_{0}^* (\alpha_0 , 0 )=q_{\infty}^* (\alpha_{\infty},0)=\frac{4N}{N-2}>4.
$$
Hence we get an existence result by choosing
$$
4<q_1 <\frac{4N}{N-2} < q_2 .
$$
\noindent Assume now $K(r)= \min \left\{  r^N , r^{2N}     \right\}$ and choose $\beta_0 =\beta_{\infty}=0$, $\alpha_0 =2N$ and $ \alpha_{\infty}=N$ 
in Theorems \ref{THM0} and \ref{THM1}. We get
$$
q_{0}^* (\alpha_0 , 0 )=\frac{6N}{N-2}, \quad q_{\infty}^* (\alpha_{\infty},0)=\frac{4N}{N-2}>4
$$
where $ q_{0}^* (\alpha_0 , 0 )> q_{\infty}^* (\alpha_{\infty},0) $, so that we can choose $q_1 = q_2 =q$ and get existence of solutions for power nonlinearities 
$g(t) = \min \{ t^{q_1 -1} , t^{q_2 -1} \}= t^{q-1}$ with
$$
4< \frac{4N}{N-2}<q<\frac{6N}{N-2}.
$$ 
\end{exa}

\begin{exa} \label{ex4}
Assume that there exist $c , \delta >0$ such that $K(r) \geq c \, e^{\delta r}$ as $r \rightarrow +\infty$. The results in \cite{Uberlandio2} cannot be applied because they require $\limsup_{r \rightarrow \infty} \frac{K(r)}{r^b} < +\infty$ for some $b \in R$. To give an explicit example, assume $K(r)= r^N  e^r$ and $V(r)= e^{2r}$. Then we take $\beta_0 =0$, $\beta_{\infty}=1/2$ and $\alpha_0 = \alpha_{\infty}=N$  in Theorems \ref{THM0} and \ref{THM1}, and we get 
$$
q_{0}^* (\alpha_0 , 0 )=\frac{4N}{N-2}> q_{\infty}^* (\alpha_{\infty},1/2)=2\, \frac{2N-1}{N-2}>4.
$$
So we can choose $q_1 = q_2 =q $ and this gives existence results for$g(t) = t^{q-1}$ with
$$
4 < 2\, \frac{2N-1}{N-2}<q<\frac{4N}{N-2}.
$$
\end{exa}

\begin{exa} \label{ex5}
Assume that $K(r) = o(r^N )$ for all $N$, as $r\rightarrow 0^+$. For example, $K(r) = c e^{-\delta /r}$ for $r$ near zero, with $c, \delta >0$. As before, the results of \cite{Uberlandio2} cannot be applied because they need a power-like behavior of $K$ near zero. Assume also that, for $r \rightarrow + \infty$, it holds $K(r) = r^{\alpha} V(r)$ for some $\alpha \in \mathbb{R}$. Notice that this does not require any specific asymptotic behavior at $\infty$ for $V$ and $K$, and the only hypothesis on the behavior of $V$ at $0$ is, again,  ${\bf (H)}$. Fix $\alpha_{\infty}= \alpha$ and $\beta_{\infty}=1$ in Theorem \ref{THM0}. Hence
$$
q_{\infty}^* (\alpha_{\infty},1)=2\, \frac{\alpha + N-2}{N-2}  = \frac{2\alpha }{N-2}+ 2.
$$
and we can choose $q_2 >\max \left\{4,   \frac{2\alpha }{N-2}+ 2  \right\}$. Once we have fixed such a $q_2$, we let $\beta_0 =0$ and $\alpha_0$ such that $2 \, \frac{\alpha_0 +N}{N-2}> q_2$. This means $q_{0}^* (\alpha_0 , 0 )>q_2$ in Theorem \ref{THM1}, so that we can take $q=q_1 =q_2$ and get an existence result with $g(t) = t^{q-1}$ for any $q >\max \left\{4,   \frac{2\alpha }{N-2}+ 2  \right\}$. 
As another example of the same kind, assume $K(r) =  e^{-1 /r}$ and 
$V(r)= 1/r^2$ for all $r>0$. It is easy to see that the best choice in Theorem \ref{THM1} is $\alpha_{\infty}= \beta_{\infty}=0$, which gives 
$$
q_{\infty}^* (\alpha_{\infty},\beta_{\infty})= \frac{ 2N}{N-2}  .
$$
As before, for any fixed $q_2 >\max \left\{4,   \frac{2N }{N-2}  \right\}$ we can let $\beta_0 =0$ and $\alpha_0$ large enough in such a way that $q_{0}^* (\alpha_0 , 0 )>q_2$, so that we can take $q_1 = q_2 =q$. Hence we get a solution for $g(t) = t^{q-1}$ with any $q >\max \left\{4,   \frac{2N }{N-2}  \right\}$. Notice that this means $q>6$ for $N=3$ and $q>4$ for $N\geq 4$.

\end{exa}

\begin{exa} \label{ex6}
Let $V(r)= 1/r^2$ for all $r>0$, and assume
$K(r)= c r^N$ for $r$ near zero ($c>0$) and  $K(r) \leq C$ for $r \rightarrow + \infty$. For example $K(r) = \min \left\{ r^N , 1 \right\}$. 
Hence the coefficients $a_0 ,b_0$ of \cite{Uberlandio2} are given by $b_0 = N$ and $a_0 = -2$, and the results of \cite{Uberlandio2} cannot be applied because they need $a_0 \geq b_0$. 
We fix $\alpha_0 = N$ and $\beta_0 = 0= \alpha_{\infty}=  \beta_{\infty}$  in Theorems \ref{THM0} and \ref{THM1}, so that
$$
q_{0}^* (\alpha_0 , \beta_0 )= q_{0}^* (N, 0 ) = \frac{4N}{N-2}>4 , \quad  q_{\infty}^* (\alpha_{\infty},\beta_{\infty})=
q_{\infty}^* (0,0)= \frac{2N}{N-2}.
$$
Hence we can take $q=q_1 = q_2$ and get existence results for power nonlinearities $g(t) = t^{q-1}$ with $q \in \left(4, \frac{4N}{N-2} \right)$ if $N\geq 4$, and $q \in (6, 12)$ if $N=3$.

\end{exa}

\begin{rem}
 As a final remark, we observe that most of existence results we can formulate for explicit potentials concern potentials $K$'s decaying fast enough as $r\rightarrow 0$. This is the major limitation of our work. Nevertheless we believe that it might be overcome by a careful analysis of our estimates, and we hope to do this in a future paper.
\end{rem}


\begin{thebibliography}{SK}

\normalsize
\baselineskip=17pt


\bibitem{AiresSouto}
Aires, J.F., Souto, M.A.:
Existence of solutions for a quasilinear Schr\"odinger equation with vanishing potentials. 
J. Math. Anal. Appl. \textbf{416}, 924-946  (2014).


\bibitem{AubinEkeland}
Aubin, J.P., Ekeland, I.,
Applied nonlinear analysis.
Courier Corporation 2006.

\bibitem{Anoop}
Anoop, T. V., Dr\'abek, P., Sasi, S.:
Weighted quasilinear eigenvalue problems in exterior domains. 
Calc. Var. Partial Differential Equations \textbf{53}, 961-975 (2015).

\bibitem{BGR_bilap}
Badiale, M., Greco, S., Rolando, S.:
Radial solutions for a biharmonic equation with vanishing or singular radial potentials. 
Nonlinear.Anal. \textbf{185}, 97-122 (2019).

\bibitem{BGR_p}
Badiale, M., Guida, M., Rolando, S.:
Compactness and existence results for the $p$-Laplace equation. 
J.Math.Anal.Appl. \textbf{451}, 345-370 (2017).

\bibitem{BGR_II}
Badiale, M., Guida, M., Rolando, S.:
Compactness and existence results in weighted Sobolev spaces of radial functions. Part II: Existence. 
Nonlinear Differ. Equ. Appl. \textbf{23}:67 (2016).

\bibitem{BGR_I}
Badiale, M., Guida, M., Rolando, S.:
Compactness and existence results in weighted Sobolev spaces of radial functions. Part I: Compactness. 
Calc. Var. Partial Differential Equations \textbf{54}, 1061-1090 (2015).

\bibitem{BPR}
Badiale, M., Pisani, L., Rolando, S.:
Sum of weighted Lebesgue spaces and nonlinear elliptic equations.
NoDEA, Nonlinear Differ. Equ. Appl. \textbf{18}, 369-405 (2011).

\bibitem{Brandi-et}
Brandi, H., Manus, C., Mainfray, G., Lehner, T., Bonnaud, G.: 
Relativistic and ponderomotive self-focusing of a laser beam in a radially inhomogeneous plasma. 
Phys. Fluids \textbf{B5}, 3539–3550 (1993).



\bibitem{BZ}
Badiale, M., Zaccagni, F.:
Radial nonlinear elliptic problems with singular or vanishing potentials.
Adv.Nonlinear Stud.\textbf{18}, 409-428 (2018).

\bibitem{Cai-Su-Sun}
Cai, H., Su, J., Sun, Y.:
Sobolev type embeddings and an inhomogeneous quasilinear elliptic equation on $\mathbb{R}^N$ with singular weights.
Nonlinear Anal. \textbf{96}, 59-67 (2014).


\bibitem{ColinJeanjean}
Colin,M., Jeanjean,L.:
Solutions for a quasilinear Schr\"{o}dinger equation: a dual approach. 
Nonlinear Anal. \textbf{56}, 213-226  (2004).



\bibitem{OMiyagakiSoares}
do O', J.M., Miyagaki,O., Soares,S.:
Soliton solutions for quasilinear Schr\"{o}dinger equations: the critical exponential case.
Nonlinear Anal.  \textbf{67} , 3357-3372  (2007).


\bibitem{Uberlandio1}
do O', J.M., Severo, U.B.:
Quasilinear Schr\"{o}dinger equations involving concave and convex nonlinearities.
Commun.Pure Apppl.Anal.  \textbf{8}, 621-644 (2009).



\bibitem{Uberlandio3}
do O', J.M., Severo, U.B.:
Solitary waves for a class of quasilinear Schr\"{o}dinger equations in dimension two.
Calc. Var. Partial Differential Equations \textbf{38} , 275-315  (2010).



\bibitem{FangSzulkin}
Fang, X.D., Szulkin, A.:
Multiple solutions for a quasilinear Schr\"{o}dinger equation.
J. Differential Equations \textbf{254}, 2015-2032  (2013).



\bibitem{FurtadoSilvaSilva}
Furtado, M.F., da Silva, E.D., Silva, M.L.:
Quasilinear elliptic problems under asymptotically linear conditions at infinity and at the origin.
Z. Angew. Math. Phys. \textbf{66}, 277-291  (2015).



\bibitem{Gloss}
Gloss, E.:
Existence and concentration of positive solutions for a quasilinear equation in $R^N$.
J. Math. Anal. Appl. \textbf{371}, 465-484  (2010).



\bibitem{GR-nls}
Guida, M., Rolando, S.:
Nonlinear Schr\"{o}dinger equations without compatibility conditions on the potentials. 
 J. Math. Anal. Appl. \textbf{439}, 347-363 (2016).

\bibitem{Kuri}
Kurihara, S.:
Large-amplitude quasi-solitons in superfluids films. 
J. Phys. Soc. Japan \textbf{50}, 3262–3267 (1981).

\bibitem{Kwon}
Kwon, O.:
Nonexistence of positive solutions for quasilinear equations with decaying potentials.
Mathematics  \textbf{8}, 425 (2020). 
https://doi.org/10.3390/math8030425


\bibitem{Li-Huang}
Li, G., Huang, Y.:
Positive solutions for critical quasilinear Schrödinger equations with potentials vanishing at infinity.
Discrete Contin. Dyn. Syst., to appear. 
https://doi.org/10.3390/math8030425


\bibitem{Ligani}
Ligani, B.:
Risultati di esistenza per alcune classi di equazioni di Schr\"{o}dinger quasilineari.
Master's degree thesis, University of Turin, 2019.


\bibitem{LiuWang}
Liu, J., Wang,Z.Q.:
Soliton solutions for quasilinear Schr\"{o}dinger equations. I.
Proc. Amer. Math. Soc. \textbf{131}, 441-448  (2003).



\bibitem{LiuLiuWang1}
Liu, X.Q., Liu,J.Q., Wang,Z.Q.:
Quasilinear elliptic equations via perturbation method.
Proc. Amer. Math. Soc. \textbf{141}, 253-263  (2013).



\bibitem{LiuLiuWang2}
Liu, X.Q., Liu,J.Q., Wang,Z.Q.: 
Quasilinear elliptic equations with critical growth via perturbation method.
J. Differential Equations \textbf{254}, 102-124  (2013).



\bibitem{LiuWangWang1}
Liu, J., Wang, Y., Wang,Z.Q.:
Soliton solutions for quasilinear Schr\"{o}dinger equations II. 
J. Differential Equations \textbf{187} , 473-493  (2003).


\bibitem{LiuWangWang2}
Liu, J., Wang, Y., Wang,Z.Q.:
Solutions for quasilinear Schr\"{o}dinger equations via the Nehari method, 
Comm. Partial Differential Equations \textbf{29}, 879-901  (2004).


\bibitem{PoppenbergSchmittWang}
Poppenberg,M., Schmitt,K., Wang,Z.Q.:
On the existence of soliton solutions to quasilinear Schr\"{o}dinger equations, 
Calc. Var. Partial Differential Equations \textbf{14}, 329-344  (2002).



\bibitem{RuizSiciliano}
Ruiz,D., Siciliano,G.:
Existence of ground states for a modified nonlinear Schr\"{o}dinger equation, 
Nonlinearity \textbf{23}, 1221-1233  (2010).


\bibitem{Uberlandio2}
Severo, U.B., G.M. de Carvalho:
Quasilinear Schr\"{o}dinger equations with unbounded or decaying potentials.
Math. Nachr. \textbf{291}, 492-517 (2018).

\bibitem{Uberlandio4}
Severo, U.B., G.M. de Carvalho:
Quasilinear Schr\"{o}dinger equations with a positive parameter and involving unbounded or decaying potentials.
Appl. Anal. \textbf{100}, 229-252 (2021).



\bibitem{SilvaVieira}
Silva, E.B., Vieira, G.F.:
Quasilinear asymptotically periodic Schr\"{o}dinger equations with critical growth.
Calc. Var. Partial Differential Equations \textbf{39}, 1-33  (2010).



\bibitem{Su12}  
Su, J.:
Quasilinear elliptic equations on $\mathbb{R}^{N}$ with singular potentials and bounded nonlinearity.
Z. Angew. Math. Phys. \textbf{63}, 51-62 (2012).

\bibitem{SuTian12}
Su, J., Tian, R.:
Weighted Sobolev type embeddings and coercive quasilinear elliptic equations on $\mathbb{R}^{N}$.
Proc. Amer. Math. Soc. \textbf{140}, 891-903 (2012).

\bibitem{Su-Wang-Will-p}
Su, J., Wang, Z.-Q., Willem, M.:
Weighted Sobolev embedding with unbounded and decaying radial potentials.
J. Differential Equations \textbf{238}, 201-219 (2007).

\bibitem{Su-Wang}
Su, J., Wang, Z.-Q.:
Sobolev type embedding and quasilinear elliptic equations with radial potentials.
J. Differential Equations \textbf{250}, 223-242 (2011).


\bibitem{YangDing}
Yang,M., Ding,Y.:
Existence of semiclassical states for a quasilinear Schr\"odinger equation with critical exponent in $R^N$.
Ann.Mat.PuraAppl. \textbf{192}, 787-809  (2013).


\bibitem{YangWangZhao}
Yang,X., Wang,W., Zhao,F.:
Infinitely many radial and non-radial solutions to a quasilinear Schr\"{o}dinger equation.
Nonlinear Anal. \textbf{114}, 158-168  (2015).


\bibitem{Yang-Zhang}
Yang, Y., Zhang, J.:
A note on the existence of solutions for a class of quasilinear elliptic equations: an Orlicz-Sobolev space setting. 
Bound. Value Probl. \textbf{2012}, 2012:136, 7 pages.


\bibitem{Zhang13}  Zhang, G.:
Weighted Sobolev spaces and ground state solutions for quasilinear elliptic problems with unbounded and
decaying potentials.
Bound. Value Probl. \textbf{2013}, 2013:189, 15 pages.

\end{thebibliography}
\end{document}